\documentclass[12pt]{amsart}  
\usepackage{amsthm, amsmath, amscd, amssymb, latexsym, stmaryrd, color}
\usepackage[top=2.5cm, bottom=2.5cm, left=2cm, right=2cm]{geometry}
\usepackage[colorinlistoftodos]{todonotes}
\theoremstyle{plain}

\newtheorem*{theorem*}{Theorem}
\newtheorem{theorem}{Theorem}[section]
\newtheorem{lem}[theorem]{Lemma}
\newtheorem{prop}[theorem]{Proposition}

\newtheorem{cor}[theorem]{Corollary}

\theoremstyle{definition}

\newtheorem{defn}[theorem]{Definition}

\newtheorem{rmk}[theorem]{Remark}
\theoremstyle{remark}

\usepackage[mathscr]{eucal}

\usepackage{graphics, graphpap}
\usepackage{array, tabularx, longtable}
\usepackage{color}
\numberwithin{equation}{section}

\usepackage{url}
\usepackage[T1]{fontenc}
\usepackage{array,epsfig}
\usepackage{amsmath}
\usepackage{old-arrows}
\usepackage{amsfonts}
\usepackage{amssymb}
\usepackage{amsxtra}
\usepackage{latexsym}
\usepackage{dsfont}
\usepackage{mathrsfs}
\usepackage{color}
\usepackage[all]{xy}
\usepackage{tensor}
\usepackage{xfrac}

\usepackage{tikz}

\makeatletter
\newcommand\mathcircled[1]{%
  \mathpalette\@mathcircled{#1}%
}
\newcommand\@mathcircled[2]{%
  \tikz[baseline=(math.base)] \node[draw,circle,inner sep=1pt,color=red] (math) {$\m@th#1#2$};%
}
\makeatother 

\usepackage{tikz-cd}
\usepackage[pagebackref=true]{hyperref}

\usepackage{graphicx}
\usepackage{mathtools}
\usepackage{caption}
\usepackage[OT2,T1]{fontenc}
\usepackage[page,toc,titletoc,title]{appendix}

\hypersetup{
 colorlinks,
 linkcolor=blue, citecolor=blue
 }

\newcommand{\id}{\operatorname{id}}

\newcommand{\pr}{\operatorname{pr}}
\newcommand{\Hom}{\operatorname{Hom}}

\newcommand{\sh}{\mathbf{SH}_{\mathfrak{M}}}
\newcommand{\shbb}{\mathbb{SH}_{\mathfrak{M}}}
\newcommand{\shct}{\mathbf{SH}_{\mathfrak{M},\textnormal{ct}}}
\newcommand{\qush}{\mathbf{QUSH}_{\mathfrak{M}}}
\newcommand{\qushct}{\mathbf{QUSH}_{\mathfrak{M},\textnormal{ct}}}

\newcommand{\Fscr}{\mathscr{F}}
\newcommand{\Gscr}{\mathscr{G}}
\newcommand{\SW}{\operatorname{SW}}
\newcommand{\Ex}{\operatorname{Ex}}

\newcommand{\gm}{\mathbb{G}}
\newcommand{\Abb}{\mathbb{A}}

\newcommand{\Cbb}{\mathbb{C}}
\newcommand{\Mscr}{\mathscr{M}}
\newcommand{\var}{\operatorname{Var}}

\newcommand{\sets}{\mathbf{Sets}}

\newcommand{\mfrak}{\mathfrak{M}}

\newcommand{\Spec}{\operatorname{Spec}}

\DeclareSymbolFont{cyrletters}{OT2}{wncyr}{m}{n}
\DeclareMathSymbol{\Sha}{\mathalpha}{cyrletters}{"58}
\DeclareMathSymbol{\Be}{\mathalpha}{cyrletters}{"42}

\title{Nearby cycles at infinity as a triangulated functor}  

\author[Khoa Bang. P]{Khoa Bang Pham}
\address{University of Rennes, Rennes, France\newline \indent 
35700, Rennes, France}
\email{khoa-bang.pham@univ-rennes1.fr}
\thanks{}


\begin{document}           
\begin{abstract}
For a polynomial function $f \colon \mathbb{C}^n \longrightarrow \mathbb{C}$, it is well-known in singularity theory (after Thom, Pham, Verdier,...) that outside a finite subset of $\mathbb{C}$, the function is a locally trivial $C^{\infty}$-fibration. The minimal such finite set is called the bifurcation set associated with $f$ and determining the bifurcation sets is a difficult task in singularity theory. In \cite{raibaut-2012} (see also, \cite{raibaut+fantini-2020}\cite{matsui+takeuchi-2013}\cite{takeuchi+tibuar-2016}), Raibaut attaches to such a function a virtual invariant called \textit{motivic nearby cycles at infinity}. This invariant lives in some Grothendieck ring of varieties and measures the difference between the Euler characteristics of the general fiber and a fixed fiber. In this work, we show that the motivic nearby cycles at infinity admits a functorial version in the context of motivic homotopy theory, called the \textit{motivic nearby cycles functors at infinity}. The motivic nearby cycles functors at infinity live in the world of motives and hence capture cohomological information (not just Euler characteristics) of singularities at infinity and realizes to Raibaut's construction in the world of virtual motives. Moreover, our construction is universal in the sense that it is applicable to any theory of nearby cycles functors defined in terms of six operations.
\end{abstract}
\maketitle                 

\section*{Introduction}

\subsection*{State of the art}
Let $U$ be a smooth integral $\mathbb{C}$-scheme of finite type and $f: U \longrightarrow \Abb^1_{\Cbb}$ be a regular dominant function. It is well-known that there exists a finite set $B$ of closed points of $\Abb_{\Cbb}^1$ such that the morphism 
\begin{equation*}
    f_{U \setminus f^{-1}(B)}: U \setminus f^{-1}(B) \longrightarrow \Abb^1_{\Cbb} \setminus B
\end{equation*}
is a locally trivial $C^{\infty}$-fibration (see for example \cite{pham-1983} or \cite{verdier-1976}). Consequently, a fiber $f^{-1}(a)$ with $a \in \mathbb{A}_{\mathbb{C}}^1 \setminus B$ of the morphism $f_{U \setminus f^{-1}(B)}$ does not depend on the choice of the point $a$, it is called the \textit{general fiber} and written as $f^{-1}(a_{\textnormal{gen}})$. The smallest such a set $B$ is called the \textit{bifurcation set}, denoted $B^{\textnormal{top}}_f$. The bifurcation set $B^{\textnormal{top}}_f$ contains the set of non-smooth closed points of $f$. However, this inclusion is strict unless we assume that $f$ is proper. For instance, let us consider the morphism $f: \mathbb{A}_{\mathbb{C}}^2 \longrightarrow \mathbb{A}_{\mathbb{C}}^1$ given by $f(x,y)=x(xy-1)$. This morphism has been studied in \cite{broughton-1983} and we see that it is smooth, however, the bifurcation set of this morphism contains a closed point corresponding to the origin. Topologically, the fiber over the origin (the special fiber) is a disjoint union of a line and a line minus a point while the general fibers are graph of function over $\mathbb{C}^{\times}$. To understand what is happening here, we have to compare the special fiber with the general fibers in a "neighborhood of infinity". In the topological situation, this means that we have to look at the fibers when $\left|x \right|^2 + \left| y \right|^2$ grows sufficiently large whereas algebraically, we should compactify $f$ by embedding it into a variety proper over $\mathbb{A}_{\mathbb{C}}^1$ to see the singularity at infinity. 

In case $U = \mathbb{A}^2_{\mathbb{C}}$, it is known that (see \cite{hhvui+ldtrang-1984}) the bifurcation set can be given by
\begin{equation*}
    B^{\mathrm{top}}_f = \left \{a \in \mathbb{A}^1_{\mathbb{C}}(\mathbb{C}) \mid \chi_c(U_a) \neq \chi_c(U_{a_{\textnormal{gen}}}) \right \},
\end{equation*}
where $U_a$ denotes the fiber over a closed point $a$, $U_{a_{\textnormal{gen}}}$ denotes the general fiber (in a topological sense), $\chi_c$ denotes the Euler characteristic with compact support. Thus, in good situation, knowing the difference between the Euler characteristic of a fiber and the Euler characteristic of the general fiber is sufficient to detect singular points at infinity (see also \cite{parusinski-1995}). In \cite{raibaut-2012}\cite{raibaut+fantini-2020} (see also \cite{matsui+takeuchi-2013}\cite{takeuchi+tibuar-2016}), the authors define a virtual invariant lying in some Grothendieck ring of varieties and measuring this difference, called \textit{motivic nearby cycles at infinity for the value $a$}, denoted $\psi_{f,a}^{\infty}$. If $a$ is a non-smooth point of $f$, then 
\begin{equation} \label{eq:1}
    \chi_c(\psi_{f,a}^{\infty}) = \chi_c(U_{a_{\textnormal{gen}}}) - \chi_c(U_a).
\end{equation}
For a proof, one can consult \cite[Theorem 3.3]{raibaut+fantini-2020}. The set 
\begin{equation*}
    B_f^{\textnormal{Rai}} =  \left \{a \in \mathbb{A}^1_k(k) \mid \psi^{\infty}_{f,a} \neq 0 \right \} \cup \mathrm{disc}(f) 
\end{equation*}
is called the \textit{motivic bifurcation set} in \cite{raibaut-2012}\cite{raibaut+fantini-2020} and contains the topological bifurcation set $B^{\textnormal{top}}_f$ in case $U = \mathbb{A}^2_{\mathbb{C}}$. Therefore, the study of singular points at infinity in some sense can be translated into the study of motivic nearby cycles at infinity. The construction of motivic nearby cycles at infinity for the value $a$ then goes as follows: if $U$ is a $\mathbb{C}$-scheme of finite type and so by the Nagata theorem (see for instance, \cite{conrad-2012}), there exists a diagram
\begin{equation*}
    \begin{tikzcd}[sep=large]
        U \arrow[r, "u",hook] \arrow[rd,"f",swap] & X \arrow[d,"\hat{f}"] \\ 
        & \mathbb{A}_{\mathbb{C}}^1
    \end{tikzcd}
\end{equation*}
with $\hat{f}$ proper and $u$ an open immersion. We form the following commutative diagram by base change
\begin{equation*}
    \begin{tikzcd}[sep=large]
     U_{\eta} \arrow[r]  \arrow[d,"u_{\eta}",swap] & U  \arrow[d,"u"] & U_{\sigma} \arrow[d,"u_{\sigma}"] \arrow[l] \\ 
        X_{\eta} \arrow[r]  \arrow[d,"\hat{f}_{\eta}",swap] & X  \arrow[d,"\hat{f}"] & X_{\sigma} \arrow[d,"\hat{f}_{\sigma}"] \arrow[l] \\ 
        \eta \coloneqq \mathbb{G}_{m,k} \arrow[r] & \mathbb{A}_k^1 & \sigma \coloneqq \Spec(k). \arrow[l]
    \end{tikzcd}
    \end{equation*} 
where $\Spec(k)$ is embedded in $\mathbb{A}^1_k$ via the zero section and $\gm_{m,k}$ is the corresponding open complement. The nearby cycles at infinity for the value $a$ is then defined as
\begin{equation*}
    \psi^{\infty}_{f,a} \coloneqq \hat{f}_{\sigma!}(\psi_{\hat{f}-a,U} - u_{\sigma!}\psi_{f-a}) \in \mathscr{M}_{\mathbb{C}}^{\hat{\mu}}
\end{equation*}
where $\psi_{f-a,U},\psi_{f-a}$ are motivic nearby cycles and $\mathscr{M}_{\mathbb{C}}^{\hat{\mu}}$ is the equivariant Grothendieck ring of varieties (see for instance, \cite{guibert+loeser+merle-2006}).  The invariant $\psi^{\infty}_{f,a}$ does not depend on the choice of a compactification (see \cite{raibaut-2012}). More importantly, the proof of formula \ref{eq:1} is based on previous results in \cite{bittner-2005}\cite{guibert+loeser+merle-2006}\cite{denef+loeser-1998} saying that the motivic nearby cycles functor is a motivic incarnation of the nearby cycles morphism
\begin{equation*} 
    \begin{tikzcd}[sep=large]
        \mathscr{M}_X \arrow[r,"\psi_f"] \arrow[d,"\mathrm{H}_X",swap] & \mathscr{M}_{X_{\sigma}}^{\hat{\mu}} \arrow[d,"\mathrm{H}_{X_{\sigma}}"] \\ 
        K_0(D^b(\mathrm{MHM}(X))) \arrow[r,"\Psi_f"] & K_0(D^b(\mathrm{MHM}(X_{\sigma}))^{\textnormal{mon}}). 
    \end{tikzcd}
\end{equation*}
where $\mathrm{H}$ is the motivic Hodge-Grothendieck characteristic (see for instance, \cite{steenbrink}) and $\Psi_f$ is the nearby cycles functor of mixed Hodge modules (see \cite{saito}). This means that 
\begin{align*}
    \mathrm{H}_{\mathbb{C}}(\psi_{f,a}^{\infty}) & = [\hat{f}_{\sigma!}\Psi_{\hat{f}-a}u_!(\underline{\mathbb{Q}}_U)] - [f_{\sigma!}\Psi_{f-a}(\underline{\mathbb{Q}}_U)] \\ 
    & = [\Psi_{t-a}(f_!\underline{\mathbb{Q}}_U)] - [f_{\sigma!}\Psi_{f-a}(\underline{\mathbb{Q}}_U)]  \ \ \ \ \ \ \text{by proper base change}
\end{align*}
where $\underline{\mathbb{Q}}_U$ is the trivial mixed Hodge module on $U$. Informally, this suggests that there should exist a categorical version of $\psi^{\infty}_{f,a}$, i.e., an actual functor denoted $\Psi^{\infty}_{f,a}$ who realizes to $\psi^{\infty}_{f,a}$ and behaves like the cone of the base change morphism 
\begin{equation*}
   f_{\sigma!}\Psi_{f-a} \longrightarrow \Psi_{t-a}f_! \longrightarrow \Psi^{\infty}_{f,a} \longrightarrow +1.
\end{equation*}
Said differently, $\Psi^{\infty}_{f,a}$ lives in the world of motives and captures cohomological information of singularities at infinity, whereas $\psi^{\infty}_{f,a}$ captures their Euler characteristic. In this work, we prove that such a functor $\Psi^{\infty}_{f,a}$ exists with the desired properties. Moreover, our construction is universal in the sense that it is applicable to any theory of nearby cycles functors defined by six functors formalisms. For instance, one can construct a functor $\Psi^{\infty,\textnormal{Hodge}}_{f,a}$ in the context of complex varieties which is the image of our functor $\Psi^{\infty}_{f,a}$ under the Betti realization (see \cite{ayoub-2010}).
\subsection*{Formulation of the main results} 
We work in the setting of motivic homotopy theory though our constructions can be carried word by word to other settings. Given a scheme $X$, we can associate with it a triangulated category $\sh(X)$, called the \textit{motivic stable homotopy category} of Voevodsky-Morel-Ayoub (see for instance \cite{voevodsky-1999}\cite{ayoub-thesis-2}\cite{ayoub-icm-2014}), where $\mfrak$ is a sufficiently good category which does not play any role in this article. Let $k$ be a field, $X$ be a $k$-variety. We attach to every $k$-morphism $f\colon X \longrightarrow \mathbb{A}_k^1$ the diagram obtained by base change
\begin{equation*}
    \begin{tikzcd}[sep=large]
        X_{\eta} \arrow[r,"j_f"]  \arrow[d,"f_{\eta}",swap] & X  \arrow[d,"f"] & X_{\sigma} \arrow[d,"f_{\sigma}"] \arrow[l,"i_f",swap] \\ 
        \eta \coloneqq \mathbb{G}_{m,k} \arrow[r,"j_{\id}"] & \mathbb{A}_k^1 & \sigma \coloneqq \mathrm{Spec}(k) \arrow[l,"i_{\id}",swap]
    \end{tikzcd}
    \end{equation*}
where $i_{\id}$ is the zero section and $j_{\id}$ is its open complement. In \cite{ayoub-thesis-2}\cite{ayoub-2007}, Ayoub developed the theory of motivic nearby cycles functor associated to the above diagram: there are triangulated functors
\begin{equation*}
    \Psi_{f}\colon \sh(X_{\eta}) \longrightarrow \sh(X_{\sigma})
\end{equation*}
equipped with several base change morphisms with respect to four operations $(g^*,g_*,g_!,g^!)$. In particular, let $g \colon Y \longrightarrow X$ be a $k$-morphism with $X,Y$ quasi-projective over $k$, then there is a base change morphism
\begin{equation*}
    \mu_g \colon g_{\sigma!}\Psi_{f \circ g}  \longrightarrow \Psi_f g_{\eta!}
\end{equation*}
which are isomorphisms if $g$ is projective (in the appendix, we strengthen projectiveness to properness).
\begin{theorem*}
    There exists a triangulated functor $\Psi^{\infty}_{f,g}\colon \sh(Y_{\eta}) \longrightarrow \sh(X_{\sigma})$ fitting into a distinguished triangles (of triangulated functors)
    \begin{equation*}
   g_{\sigma!}\Psi_{f \circ g} \longrightarrow \Psi_{f}g_{\eta!} \longrightarrow \Psi^{\infty}_{f,g} \longrightarrow +1.
\end{equation*}
In particular, for every morphism $f \colon X \longrightarrow \mathbb{A}_k^1$, there exists a triangulated functor $ \Psi^{\infty}_{\id,f}\colon \sh(X_{\eta}) \longrightarrow \sh(k)$ fitting into a distinguished triangles (of triangulated functors)
    \begin{equation*}
   f_{\sigma!}\Psi_{f} \longrightarrow \Psi_{\id}f_{\eta!} \longrightarrow \Psi^{\infty}_{\id,f} \longrightarrow +1.
\end{equation*}
The functor $\Psi^{\infty}_{\id,f}$ is called the motivic nearby cycles functor at infinity of $f$. 
\end{theorem*}
Moreover, it is similar to the fact that the nearby cycles $\psi_f$ constructed by Denef and Loeser (see for instance, \cite{denef+loeser-1998}\cite{guibert+loeser+merle-2006}) is the non-virtual incarnation of $\Psi_f$, we prove here that Raibaut's nearby cycles at infinity $\psi_{f}^{\infty}$ studied in \cite{raibaut-2012}\cite{raibaut+fantini-2020}\cite{matsui+takeuchi-2013}\cite{takeuchi+tibuar-2016} is the non-virtual incarnation of $ \Psi^{\infty}_{\id,f}$ (see proposition \ref{prop: motivic incarnation}, corollary \ref{cor: raibaut's theorem 2.9 analogue} and \ref{prop: compatibility of nearby functors with nearby cycles} for more precise statements)
\begin{theorem*} 
Let $X$ be a $k$-variety and $f\colon X \longrightarrow \mathbb{A}^1_k$ be a morphism of $k$-varieties, then the diagram
   \begin{equation*}
        \begin{tikzcd}[sep=large]
            \Mscr_X \arrow[r,"\psi_f"] \arrow[d,"\chi_{X,c}",swap]& \Mscr_{X_{\sigma}} \arrow[d,"\chi_{X_{\sigma},c}"] \\ 
            K_0(\shct(X)) \arrow[r,"\Psi_f"] & K_0(\shct(X_{\sigma}))
        \end{tikzcd}
    \end{equation*}
is commutative. This implies the identity
    \begin{equation*}
    \Psi_{\id,f}^{\infty}(\mathds{1}_{U_{\eta}}) = \psi^{\infty}_{f}
    \end{equation*}
    in $K_0(\shct(k))$. In particular, consider the Euler bifurcation set
    \begin{equation*}
          B^{\textnormal{Eu}}_f = \left \{a \in \mathbb{A}^1_k(k) \mid [\Psi^{\infty}_{\id,f-a}(\mathds{1})] \neq 0 \right \} \cup \mathrm{disc}(f) 
    \end{equation*}
is contained in $B^{\textnormal{Rai}}_f$ and hence finite. 
\end{theorem*}
Finally, we reinforce the idea that the functors $\Psi_{\id,f-a}^{\infty}$ (with $a \in \mathbb{A}_k^1)$ give the correct invariant in the world of motives for bifurcation sets by calcuting some specific examples. We are interested in an enlargement of $B^{\textnormal{Eu}}_f$, the cohomological bifurcation set $B^{\textnormal{coh}}_f = \left \{a \in \mathbb{A}^1_k(k) \mid \Psi^{\infty}_{\id,f-a}(\mathds{1}) \not\simeq 0 \right \} \cup \mathrm{disc}(f)$. We have the following result
\begin{theorem*}
Let $X$ be a $k$-variety and $f \colon X \longrightarrow \mathbb{A}_k^1$. If $a \in \mathbb{A}_k^1(k)$ is a typical value of $f$ such that $f^{-1}(a)$ is a weakly elementary fibration (in the sense of definitions \ref{defn: fibration},\ref{defn: weakly elementary fibration}) (typical examples of such fibers include basic schemes such as (finite, smooth extensions of) $\mathbb{A}^n_k \setminus \left\{ \textnormal{finite points} \right \}$, $\mathbb{G}_{m,k}^n$ and so on), then
\begin{equation*} 
\Psi_{\id,f-a}^{\infty}(f-a)^*_{\eta} \simeq 0.
\end{equation*} 
In particular, $X = \mathbb{A}_k^n$ and $f(x_1,...,x_n) = x_1 \cdots x_i$ (with $1 \leq i \leq n)$, then 
\begin{equation*}
    \left \{a \in \mathbb{A}_k^1(k) \mid \Psi^{\infty}_{\id,f-a}(\mathds{1}) \not\simeq 0 \right \} = \varnothing. 
\end{equation*}
On the other hand, in the setup of Broughton's examples, if $X = \mathbb{A}_k^2$ and $f(x,y) = x(xy-1)$ is the Broughton example, then
\begin{equation*}
    \left \{a \in \mathbb{A}_k^1(k) \mid \Psi^{\infty}_{\id,f-a}(\mathds{1}) \not\simeq 0 \right \} = \left \{0 \right \}. 
\end{equation*}
\end{theorem*}
\subsection*{Structure of the manuscript} In section {\color{blue}\S 1}, we present an overview of the formalism of six operations in motivic homotopy theory and some preliminary results for the motivic stable homotopical algebraic derivators constructred by Ayoub in \cite{ayoub-thesis-1}\cite{ayoub-thesis-2}. In section {\color{blue}\S 2.1}, we introduce the motivic nearby cycles functors of Ayoub constructed in \cite{ayoub-thesis-2} and start to construct the motivic nearby cycles functors at infinity. The proofs of the existence and the uniqueness of motivic nearby cycles at infinity are divided into two sections {\color{blue}\S 2.2} and {\color{blue}\S 2.3}. At the end of {\color{blue}\S 2}, we study the interaction of these functors with respect to quasi-unipotent motives. In section {\color{blue} \S 3}, we demonstrate that motivic nearby cycles functors of Ayoub are compatible with nearby cycles of Denef-Loeser. As a consequence, we obtain the desired result on the compatibility of motivic nearby cycles at infinity and our motivic nearby cycles functors at infinity. In the last section {\color{blue} \S 4}, we discuss several bifurcation sets and give a simple criterion to detect typical values of a function and show that the bifurcation sets of some simple monomials and Broughton's example reduce to one-point set. 

In the appendix, we attempt to remove the quasi-projectiveness hypothesis in the base change morphisms of motivic nearby cycles functors. This is a natural step though it is not completely obvious and more importantly, it is necessary in our construction of nearby cycles functors at infinity. Nearby cycles functors at infinity are defined via choices of Nagata compactifications

\section*{Conventions and notation} For a scheme $S$, a $S$-\textit{variety} is a separated $S$-scheme of finite type, we denote by $\var_S$ the category of $S$-varieties whose objects are $S$-varieties and morphisms are $S$-morphisms.

In this article, $k$ is a field of characteristic zero. For every $k$-scheme $S$, we denote by $\gm_{m,S} = S \times_k \gm_{m,k}$ the multiplicative group over $S$ and by $1 \colon S \longrightarrow \gm_{m,S}$ the unit section. When we have to fix a global coordinate, we write $\gm_{m,S}=\Spec(\mathcal{O}_S[T,T^{-1}])$ for the global spectrum of the quasi-coherent $\mathcal{O}_S$-algebra $\mathcal{O}_S[T,T^{-1}]$. If $f\colon S \longrightarrow S'$ is a morphism of $k$-schemes, we denote by $\breve{f}\colon \gm_{m,S} \longrightarrow \gm_{m,S'}$ the morphism obtained by base change. 

Let $X$ be a reduced variety over $k$ and $F$ be a closed subscheme of $X$ such that $F$ contains the singular locus of $X$ and $F$ is of codimension $\geq 1$ everywhere. By a \textit{log-resolution} $h \colon (Z,E) \longrightarrow (X,F)$, we mean a proper morphism $h \colon Z \longrightarrow X$ with $Z$ smooth over $k$ such that $h^{-1}(F)=E$ and the restriction $Z \setminus E \longrightarrow X \setminus F$ is an isomorphism and $h^{-1}(F)$ is a divisor with normal crossings. 

In this article, by the \textit{interchange law}, we mean the following statement: let $\eta \colon F \longrightarrow G$ be a natural transformation between functors $F,G \colon \mathcal{A} \longrightarrow \mathcal{B}$ and $\epsilon \colon J \longrightarrow K$ be a natural transformation between functors $J,K \colon \mathcal{B} \longrightarrow \mathcal{C}$, then the following diagram is commutative
    \begin{equation*}
        \begin{tikzcd}[sep=large]
            JF \arrow[r,"J\eta"] \arrow[d,"\epsilon F",swap] & JG \arrow[d,"\epsilon G"] \\ 
            KF \arrow[r,"K\eta"] & KG.
        \end{tikzcd}
    \end{equation*}

In this article, we carry out a lot of computations by using various base change isomorphisms so in order to make the proof easier to follow, in complicated computations, we color {\color{blue}terms} that are about to be changed in the next step. 
\section{Review of six functors formalism}
\subsection{The formalism of six operations} Let $X$ be a scheme, we shall consider the motivic stable homotopy category $\sh(X)$ of Morel-Voevodsky-Ayoub, where $\mfrak$ is a category of coefficients in the sense of \cite{ayoub-thesis-2}. For example, the category $\mfrak$ can be the category $\mathbf{Spect}^{\Sigma}_{S^1}(\Delta^{\mathrm{op}}\sets_{\bullet})$ of symmetric $S^1$-spectra of pointed simplicial sets, giving us the motivic stable homotopy category $\mathbf{SH}$ (see for instance, \cite{voevodsky-1998}) or the category $\mathbf{Compl}(\Lambda)$ of chain complexes of $\Lambda$-modules over some ring $\Lambda$, giving the category of \'etale motives (see for instance, \cite{ayoub-2014}). The category $\sh(X)$ is a unital, symmetric, closed monoidal triangulated category. We denote its unit object for the monoidal structure by $\mathds{1}_X$ but sometimes we can just write $\mathds{1}_X$ if the scheme $X$ is understood. In this paper, we heavily use the formalism of six operations $(f^*,f_*,f_!,f^!,\otimes, \underline{\Hom})$ developed by Ayoub in \cite{ayoub-thesis-1} and revisited by Cisinski and Déglise in \cite{cisinski+deglise-2019}. For any morphism of $k$-varieties $f\colon X \longrightarrow Y$, we associate with it an adjunction 
\begin{equation*}
    (f^* \dashv f_*)\colon \sh(Y) \longrightarrow \sh(X).
\end{equation*}
If $f$ is smooth, then $f^*$ admits a left adjoint
\begin{equation*}
    (f_{\#} \dashv f^*)\colon \sh(X) \longrightarrow \sh(Y).
\end{equation*}
There is an exceptional adjunction if $f$ is separated of finite type
\begin{equation*}
    (f_! \dashv f^!)\colon \sh(X) \longrightarrow \sh(Y).
\end{equation*}
To any cartesian square of $k$-varieties
\begin{equation*}
    \begin{tikzcd}[sep=large]
         X' \arrow[d,"f'",swap] \arrow[dr,phantom,"\textnormal{(C)}"] \arrow[r,"g'"] & X \arrow[d,"f"] \\ 
            Y' \arrow[r,"g"] & Y
    \end{tikzcd}
\end{equation*}
are associated with several base change morphisms such as 
\begin{align*}
     \Ex^*_*(\textnormal{C}) & \colon g^*f_* \longrightarrow  f_*'g^{'*} \\ 
     \Ex_{\#}^*(\textnormal{C}) &\colon g'_{\#}f^{'*}  \longrightarrow f^*g_{\#} \\
     \Ex_!^*(\textnormal{C}) & \colon g^*f_! \longrightarrow f'_!g^{'*} \\
     \Ex_{\#*}(\textnormal{C}) & \colon g_{\#}f'_*  \longrightarrow f_*g'_{\#} \\ 
     \Ex_{\# !}(\textnormal{C}) & \colon g_{\#}f'_!  \longrightarrow f_! g'_{\#} \\ 
     \Ex^{!*}(\textnormal{C}) & \colon  g^{'*}f^! \longrightarrow f^{'!}g^*
\end{align*}
where in $\Ex^*_{\#}(\textnormal{C}),\Ex_{\#*}(\textnormal{C}),\Ex_{\#!}(\textnormal{C})$, we require $g$ to be smooth. 

\begin{prop} 
The morphism $\Ex^*_*(\textnormal{C})$ is an isomorphism if either $g$ is smooth or $f$ is proper. The morphisms $\Ex_{\#}^*(\textnormal{C}),\Ex_{\#!}(\textnormal{C})$ are always isomorphisms. Consequently, the morphism $\Ex_{\#*}(\textnormal{C})$ is an isomorphism if $f$ is proper. The morphism $\Ex^{!*}(\textnormal{C})$ is an isomorphism if $f$ is smooth. 
\end{prop} 
\begin{proof}
    These results can be found in \cite{cisinski+deglise-2019} for instance. 
\end{proof}

\subsection{Algebraic derivators}

 For our later purposes, we have to extend the $2$-functor $\sh(-)$ to a so-called \textit{stable homotopical algebraic derivator}, see \cite[Definition 2.4.13]{ayoub-thesis-1}. A \textit{diagram of $k$-varieties} $(\Fscr,I)$ consists of a small category $I$ and a functor $\Fscr \colon I \longrightarrow \var_k$. A morphism $(f,\alpha)\colon (\Gscr,J) \longrightarrow (\Fscr,I)$ consists of a functor $\alpha\colon J \longrightarrow I$ together with a natural transformation $f\colon \Gscr \longrightarrow \Fscr \circ \alpha$. If $\alpha = \id$, we simply write $f \colon (\Gscr,I) \longrightarrow (\Fscr,I)$. For every $k$-variety $X$, we use the notation $(X,I)\colon I \longrightarrow \var_k$ to indicate the constant $X$-value diagram, i.e. it maps any $i \in I$ to $X$ and any morphism $i \longrightarrow j$ in $I$ to the identity of $X$. A morphism $(f,\alpha)\colon (\Gscr,J) \longrightarrow (\Fscr,I)$ \textit{has property $(P)$ argument by argument} (where $(P)$ is a property of morphisms of schemes, e.g. separated, proper, smooth,...) if for every $j \in J$, the morphism $\Gscr(j) \longrightarrow \Fscr(\alpha(j))$ has property $(P)$. The stable homotopical algebraic derivator $\shbb(-,-)$ assigns any diagram of $S$-varieties $(\mathscr{F},I)$ a unital symmetric monoidal triangulated category $\shbb(\mathscr{F},I)$. For any morphism $(f,\alpha) \colon (\mathscr{G},J) \longrightarrow (\mathscr{F},I)$, we associate with it an adjunction
\begin{equation*}
    ((f,\alpha)^* \dashv (f,\alpha)_*)\colon \shbb(\mathscr{F},I) \longrightarrow \shbb(\mathscr{G},J).
\end{equation*}
If $(f,\alpha)$ is smooth argument by argument, then $(f,\alpha)^*$ admits a left adjoint
\begin{equation*}
    ((f,\alpha)_{\#} \dashv (f,\alpha)^*)\colon\shbb(\mathscr{G},J) \longrightarrow \shbb(\mathscr{F},I). 
\end{equation*}
Let $\mathbf{e}$ denotes the category with a single object and a single morphism, then by definition $\shbb(X,\mathbf{e}) = \sh(X)$ for every $k$-variety $X$. Let $z \colon (\mathscr{Z},I) \longrightarrow (\mathscr{X},I)$ be a cartesian morphism closed argument by argument, we denote by $u \colon (\mathscr{U},I) \longrightarrow (\mathscr{X},I)$ be the open complement argument by argument. There are two localization sequences generalizing the localization sequences in the case of schemes
\begin{equation*}
    u_{\#}u^* \longrightarrow \id \longrightarrow z_*z^* \longrightarrow +1  \ \ \text{and} \ \ 
     z_*z^! \longrightarrow \id \longrightarrow u_*u^* \longrightarrow +1 
\end{equation*}
by \cite[Proposition 2.4.25]{ayoub-thesis-1}. To any cartesian square of diagrams of $k$-varieties
     \begin{equation*}
    \begin{tikzcd}[sep=large]
        (\mathscr{G}',J) \arrow[r,"(g'\text{,}\beta)"] \arrow[dr,phantom,"\textnormal{(C)}"] \arrow[d,"f'",swap] & (\mathscr{G},I) \arrow[d,"f"] \\ 
        (\mathscr{F}',J) \arrow[r,"(g\text{,}\beta)"] & (\mathscr{F},I).
    \end{tikzcd}
\end{equation*}
are associated with several base change morphisms such as 
\begin{align*}
     \Ex^*_*(\textnormal{C}) & \colon (g,\beta)^*f_* \longrightarrow  f_*'(g',\beta)^* \\ 
     \Ex_{\#}^*(\textnormal{C}) &\colon (g',\beta)_{\#}f^{'*}  \longrightarrow f^*(g,\beta)_{\#} \\
     \Ex_{\#*}(\textnormal{C}) & \colon (g,\beta)_{\#}f'_*  \longrightarrow f_*(g',\beta)_{\#} 
\end{align*}
where in the two last morphisms, we require $(g,\beta)$ to be smooth argument by argument. 
\begin{prop} \label{proper base change for derivators}
The base change morphism $\Ex^*_{\#}(\textnormal{C})$ is an isomorphism if $(g,\beta)$ is smooth argument by argument. The base change morphism $\Ex_*^*(\textnormal{C})$ is an isomorphism either if $f$ is proper argument by argument or $(g,\beta)$ is smooth argument by argument. 
\end{prop}
\begin{proof}
Regarding $\Ex^*_{\#}(\textnormal{C})$ and $\Ex^*_*(\textnormal{C})$, if $(g,\beta)$ is smooth argument by argument, we conclude by \cite[Proposition 2.4.20]{ayoub-thesis-1} and \cite[Corollaire 2.4.24]{ayoub-thesis-1}. In case $f$ is proper argument by argument, we proceed by repeating the proof of \cite[Théorème 2.4.22]{ayoub-thesis-1} to reduce to the case of schemes and conclude by \cite[Propostion 2.3.11]{cisinski+deglise-2019}. 
\end{proof}

\begin{prop} \label{direct images commute with homotopy colimits}
    Let $I$ be an arbitrary diagram and $g \colon Y \longrightarrow X$ be a morphism of $k$-varieties. We form the cartesian square
    \begin{equation*}
        \begin{tikzcd}[sep=large]
            (Y,I) \arrow[r,"p_I"] \arrow[d,"g",swap] & Y \arrow[d,"g"] \\ 
            (X,I) \arrow[r,"p_I"] & X
        \end{tikzcd}
    \end{equation*}
    where $p_I$ denotes the canonical projection. If $g$ is proper, then the base change morphism 
    \begin{equation*}
        \Ex_{\#*} \colon (p_I)_{\#}g_* \longrightarrow g_*(p_I)_{\#}
    \end{equation*}
    is an isomorphism. 
\end{prop}

\begin{proof}
    By \cite[Proposition 2.4.46]{ayoub-thesis-1}, we see that $\Ex_{\#*}$ is an isomorphism if $g$ is projective (note that the proof of \cite[Proposition 2.4.46]{ayoub-thesis-1} does not use the assumption on quasi-projectiveness). To handle the general case, we proceed by noetherian induction and the Chow lemma. We consider the following property of proper closed subschemes of $Y$: \medskip \\ 
      \begin{itshape} 
    Let $t \colon T \longhookrightarrow Y$ be a proper closed subscheme, if for any proper closed subscheme $z \colon Z \longhookrightarrow T$, $(p_I)_{\#}(g \circ t \circ z)_* \longrightarrow (g \circ t \circ z)_*(p_I)_{\#}$ is an isomorphism, then $(p_I)_{\#}(g \circ t)_* \longrightarrow (g \circ t)_*(p_I)_{\#}$ is an isomorphism. \medskip \\ 
    \end{itshape}
    Since $g \circ t$ is proper, by the Chow lemma, there exists a projective morphism $p \colon V \longrightarrow T$ such that $g \circ t \circ p$ is also projective and there exists an open dense subscheme $u \colon U \longhookrightarrow T$ of $T$ over which $p$ is an isomorphism.  Let $z\colon Z \longhookrightarrow T$ be the (reduced) closed complement of $U$ in $T$. Thanks to the localization sequence
    \begin{equation*}
        z_*z^! \longrightarrow \id \longrightarrow u_*u^* \longrightarrow +1 
    \end{equation*}
    we can reduce to proving 
    \begin{align*}
        (p_I)_{\#}(g \circ t \circ z)_* & \longrightarrow (g \circ t)_*(p_I)_{\#}z_* \\
        (p_I)_{\#}(g \circ t \circ u)_* & \longrightarrow (g \circ t)_*(p_I)_{\#}u_* 
    \end{align*}
    are isomorphisms. The first case follows from the hypothesis and projective case (note that a closed immersion is projective by definition)
    \begin{align*}
        (p_I)_{\#}(g \circ t \circ z)_* & \simeq (g \circ t \circ z)_*(p_I)_{\#} \\ 
        & \simeq (g \circ t)_*(p_I)_{\#}z_*.
    \end{align*}
    For the second case, we form the cartesian square
    \begin{equation*}
         \begin{tikzcd}[sep=large]
          p^{-1}(U) \arrow[r,"j"] \arrow[d,"q",swap] & V \arrow[d,"p"] \\ 
            U \arrow[r,"u"] & T.
        \end{tikzcd}
        \end{equation*}
        Since $q$ is an isomorphism, we deduce that $u_*=p_*j_*q^*$. Therefore, by the projective case, we see that
        \begin{align*}
            (p_I)_{\#}(g \circ t \circ u)_* & \simeq (p_I)_{\#}(g \circ t \circ p)_*j_*q^* \\ 
            & \simeq (g \circ t \circ p)_*(p_I)_{\#}j_*q^* \\ 
            & \simeq (g \circ t)_*(p_I)_{\#}p_*j_*q^* \\ 
            & \simeq (g \circ t)_*(p_I)_{\#}u_*
        \end{align*}
        as desired. We finish by noetherian induction. We define 
        \begin{equation*}
            \mathscr{S} = \left \{t\colon T \longhookrightarrow Y \mid T \ \text{nonempty, proper closed s.t} \ (p_I)_{\#}(g \circ t)_* \not\simeq (g \circ t)_*(p_I)_{\#}  \right \}.  
        \end{equation*}
        If $\mathscr{S}$ is not empty, then it contains an infinite descending chain, which contradicts the noetherian property of $Y$.
\end{proof}

\begin{rmk}
    It is worth noting that for $\shbb$, corollary \ref{direct images commute with homotopy colimits} holds even without the properness of $g$ (see \cite[Proposition C.12]{hoyois-2014}). Indeed, $(p_I)_{\#}$ is nothing but homotopy colimits and $g^*$ preserves compact generators so its right adjoint commutes with homotopy colimits. The proof above is more direct and avoids the use of compactly generated property and can be extended immediately to any stable homotopical algebraic derivator in the sense of \cite[Definition 2.4.13]{ayoub-thesis-1}.
\end{rmk}

\section{Motivic nearby cycles functors at infinity}
\subsection{Motivic nearby cycles functors}
Given a $k$-variety $X$ and a morphism of $k$-varieties $f\colon X \longrightarrow \mathbb{A}_k^1$, we consider the following commutative diagram obtained by base change
\begin{equation*}
    \begin{tikzcd}[sep=large]
        X_{\eta} \arrow[r,"j_f"]  \arrow[d,"f_{\eta}",swap] & X  \arrow[d,"f"] & X_{\sigma} \arrow[d,"f_{\sigma}"] \arrow[l,"i_f",swap] \\ 
        \eta \coloneqq \mathbb{G}_{m,k} \arrow[r,"j_{\id}"] & \mathbb{A}_k^1 & \sigma \coloneqq \Spec(k). \arrow[l,"i_{\id}",swap]
    \end{tikzcd}
    \end{equation*} 
In \cite{ayoub-thesis-2} (see also \cite{ayoub-2007}), Ayoub associates with this diagram a triangulated functor
\begin{equation*}
    \Psi_f\colon \sh(X_{\eta}) \longrightarrow \sh(X_{\sigma}),
\end{equation*}
called the \textit{motivic nearby cycles functor}. Let us recall the detailed construction of $\Psi_f$: let $\Delta$ be the category of finite ordinals $\mathbf{n} = \left \{0 < 1 < \cdots < n \right \}$ and  $\mathbb{N}^{\times}$ be the set of non-zero natural numbers viewed as a category whose objects are non-zero natural numbers and morphisms are defined via the opposite of the division relation. In \cite[Definition 3.5.1]{ayoub-thesis-1}, Ayoub defines a diagram of $\gm_{m,k}$-schemes $\theta_{\id} \colon(\mathscr{R},\Delta \times \mathbb{N}^{\times}) \longrightarrow \gm_{m,k}$ such that $\mathscr{R}(\mathbf{n},r) = \gm_{m,k} \times_k(\gm_{m,k})^n$. The structural morphism is given by the composition
\begin{equation*}
    \gm_{m,k} \times_k (\gm_{m,k})^n \overset{\pr_1}{\longrightarrow} \gm_{m,k} \overset{e_r}{\longrightarrow} \gm_{m,k}, 
\end{equation*}
in which $\pr_1$ is the projection on the first factor and the second one is the $r$-power morphism. We form a diagram
    \begin{equation*}
    \begin{tikzcd}[column sep = 0.8cm, row sep=1.4cm]
       (\mathscr{R} \times_{\gm_{m,k}} X_{\eta},I) \arrow[d,"f_{\eta}",swap]  \arrow[r,"\theta_f"] & (X_{\eta}, I)  \arrow[rd,"p_{I}"]  \arrow[d,"f_{\eta}"] \arrow[rr,"j_f"] & & (X,I)  \arrow[rd,"p_{I}"] & & (X_{\sigma},I) \arrow[ll,"i_f",swap] \arrow[rd,"p_{I}"] & \\ 
       (\mathscr{R},I) \arrow[r,"\theta_{\id}"] & (\gm_{m,k},I) \arrow[rd,"p_{I}",swap] & X_{\eta}  \arrow[rr,"j_f"] \arrow[d,"f_{\eta}"] & & X \arrow[d,"f"] &  &  X_{\sigma} \arrow[ll,"i_f",swap] \arrow[d,"f_{\sigma}"] & \\ 
       & & \gm_{m,k} \arrow[rr,"j_{\id}"] & & \mathbb{A}_k^1 & & \Spec(k)\arrow[ll,"i_{\id}",swap]
    \end{tikzcd}
\end{equation*}
where $\theta_f$ is the canonical projection obtained by base change and $I = \Delta \times \mathbb{N}^{\times}$. The motivic nearby cycles functor $\Psi_f\colon \sh(X_{\eta}) \longrightarrow \sh(X_{\sigma})$ is defined to be the composition
\begin{equation*}
     \Psi_f \coloneqq (p_{\Delta \times \mathbb{N}^{\times}})_{\#} (i_f)^*(j_f \circ \theta_f)_*(\theta_f)^*(p_{\Delta \times \mathbb{N}^{\times}})^*.
\end{equation*}
To shorten the notation, we set $\mathscr{R}_X \coloneqq \mathscr{R} \times_{\gm_{m,k}} X_{\eta}$ and $I = \Delta \times \mathbb{N}^{\times}$. The collection of functors $\Psi_?$ come with several base change morphisms: if $g\colon Y \longrightarrow X$ is another morphism of $k$-varieties, then there are four base change morphisms 
\begin{equation*}
      \begin{tikzcd}[sep=large]
         \sh(X_{\eta}) \arrow[r,"\Psi_f"] \arrow[d,"g_{\eta}^*",swap] & \sh(X_{\sigma}) \arrow[Rightarrow, shorten >=27pt, shorten <=27pt, dl,"\alpha_g"] \arrow[d,"g_{\sigma}^*"] \\ 
         \sh(Y_{\eta}) \arrow[r,"\Psi_{f \circ g}"] &  \sh(Y_{\sigma}). 
    \end{tikzcd} \ \ \ \  \begin{tikzcd}[sep=large]
         \sh(Y_{\eta}) \arrow[r,"\Psi_{f \circ g}"] \arrow[d,"g_{\eta*}",swap] & \sh(Y_{\sigma}) \arrow[d,"g_{\sigma*}"] \\ 
         \sh(X_{\eta}) \arrow[Rightarrow, shorten >=27pt, shorten <=27pt, ur,"\beta_g"] \arrow[r,"\Psi_f"] &  \sh(X_{\sigma})
         \end{tikzcd} 
\end{equation*}
\begin{equation*}
\begin{tikzcd}[sep=large]
         \sh(Y_{\eta}) \arrow[r,"\Psi_{f \circ g}"] \arrow[d,"g_{\eta!}",swap] & \sh(Y_{\sigma}) \arrow[Rightarrow, shorten >=27pt, shorten <=27pt, dl,"\mu_g"] \arrow[d,"g_{\sigma!}"] \\ 
         \sh(X_{\eta}) \arrow[r,"\Psi_{f}"] &  \sh(X_{\sigma}).
    \end{tikzcd} \ \ \ \ \begin{tikzcd}[sep=large]
         \sh(X_{\eta}) \arrow[d,"g_{\eta}^!",swap]\arrow[r,"\Psi_{f}"] & \sh(X_{\sigma}) \arrow[d,"g_{\sigma}^!"]   \\ 
         \sh(Y_{\eta})   \arrow[r,"\Psi_{f \circ g}"] \arrow[Rightarrow, shorten >=27pt, shorten <=27pt, ur,"\nu_g"] &  \sh(Y_{\sigma}). 
    \end{tikzcd} 
    \end{equation*} 
If $X,Y$ are quasi-projective $k$-varieties, then in \cite{ayoub-thesis-2}, it is showed that $\alpha_g,\nu_g$ are isomorphisms provided that $g$ is smooth and $\beta_g,\mu_g$ are isomorphisms provided that $g$ is projective. In the appendix of this work, we remove the assumption of quasi-projectiveness and prove that $\beta_g,\mu_g$ are isomorphisms if $g$ is proper. We note that this elimination is very natural but it is not obvious and it is important in our construction of nearby cycles functors at infinity as their definition relies on the choice of a Nagata compactification and one knows that in general, we can not factor a morphism of $k$-varieties as a composition of a projective morphism and an open immersion. 
\subsection{Existence of motivic nearby cycles functors at infinity} The starting point for the definition of nearby cycles functors at infinity is the following lemma, which shows that the cone of the base change morphism $\Ex_{\#*}$ is a triangulated functor.
\begin{lem} \label{base change criterion}
    Let $I$ be an arbitrary diagram. For any cartesian square
    \begin{equation*}
    \begin{tikzcd}[sep=large]
        (\mathscr{G},I) \arrow[r,"\pi'"] \arrow[dr,phantom,"\textnormal{(C)}"] \arrow[d,"u'",swap] & (Y,I) \arrow[d,"u"] \\ 
        (\mathscr{F},I) \arrow[r,"\pi"] & (\overline{Y},I)
    \end{tikzcd}
\end{equation*}
of diagrams of $k$-varieties where $u$ is an open immersion, denote by $v\colon (\overline{Y}\setminus Y,I) \longrightarrow (\overline{Y},I)$ its closed complement equipped with the reduced structure, then there exists a triangulated functor \begin{equation*}
    \SW(u,\pi)\colon \shbb(\mathscr{G},I) \longrightarrow \shbb(\overline{Y},I)
    \end{equation*}
    fitting into a distinguished triangle 
    \begin{equation*}
        u_{\#}\pi_*' \overset{\Ex_{\#*}(\textnormal{C})}{\longrightarrow} \pi_*u'_{\#} \longrightarrow \SW(u,\pi) \longrightarrow +1.
    \end{equation*}
\end{lem}

\begin{proof}
   By \cite[Proposition 2.4.25]{ayoub-thesis-1}, there exists a distinguished triangle
    \begin{equation*}
        u_{\#}u^* \longrightarrow \id \longrightarrow  v_{*}v^* \longrightarrow +1.
    \end{equation*}
    We apply $\pi_*u'_{\#}$ to the triangle from the right and then use \cite[Proposition 2.4.20]{ayoub-thesis-1} to obtain an isomorphism $u^*\pi_* \overset{\sim}{\longrightarrow} \pi'_*u^{'*}$ and then use smooth base change to see that $u^{'*}u'_{\#} \overset{\sim}{\longrightarrow} \id$. This yields a distinguished triangle 
    \begin{equation*}
     u_{\#}\pi_*' \longrightarrow \pi_*u'_{\#} \longrightarrow  v_{*}v^*\pi_*u'_{\#} \longrightarrow +1.
     \end{equation*}
   Finally, we define $\SW(u,\pi)$ to be $v_{*}v^*\pi_*u'_{\#}$. 
\end{proof}

Let $f \colon X \longrightarrow \mathbb{A}_k^1$, we let $\Psi_f \colon \sh(X_{\eta}) \longrightarrow \sh(X_{\sigma})$ denote its motivic nearby cycles functor. For an explicit definition and properties, we advise the reader to have a look at the appendix. 
\begin{theorem}  \label{existence of nearby cycles functors at infty}
         Let $X,Y$ be $k$-varieties and $g\colon Y \longrightarrow X$, $f\colon X \longrightarrow \mathbb{A}_k^1$ be morphisms of $k$-varieties. Then there exists a triangulated functor 
    \begin{equation*}
           \Psi^{\infty}_{f,g}\colon \sh(Y_{\eta}) \longrightarrow \sh(X_{\sigma})
    \end{equation*}
        fitting into a distinguished triangle 
    \begin{equation*}
      g_{\sigma!}\Psi_{f \circ g} \overset{\mu_g}{\longrightarrow} \Psi_{f}g_{\eta!} \longrightarrow \Psi^{\infty}_{f,g} \longrightarrow +1.
    \end{equation*}
    \end{theorem}
    \begin{proof}
We consider a compactification of $g$
        \begin{equation*}
    \begin{tikzcd}[sep=large]
        Y \arrow[r,"u"] \arrow[rd,"g",swap] & \overline{Y} \arrow[d,"\overline{g}"] \\ 
        & X.
    \end{tikzcd}
    \end{equation*}
   We apply lemma \ref{base change criterion} to the diagram 
   \begin{equation*}
    \begin{tikzcd}[sep=large]
        (\mathscr{R} \times_{\gm_{m,k}} Y_{\eta},\Delta \times \mathbb{N}^{\times}) \arrow[r,"\theta_{f \circ g}"] \arrow[d,"u_{\eta}",swap] & (Y_{\eta},\Delta \times \mathbb{N}^{\times}) \arrow[r,"j_{f \circ g}"] \arrow[d,"u_{\eta}",swap] & (Y,\Delta \times \mathbb{N}^{\times}) \arrow[d,"u"] \\ 
         (\mathscr{R} \times_{\gm_{m,k}} \overline{Y}_{\eta},\Delta \times \mathbb{N}^{\times}) \arrow[r,"\theta_{f \circ \overline{g}}"] & (\overline{Y}_{\eta},\Delta \times \mathbb{N}^{\times}) \arrow[r,"j_{f \circ \overline{g}}"] & (\overline{Y},\Delta \times \mathbb{N}^{\times})
    \end{tikzcd}
\end{equation*}
to obtain a distinguished triangle
\begin{equation*}
    \begin{tikzcd}[sep=large]
        u_{\#}(j_{f \circ g} \circ \theta_{f \circ g})_{*} \arrow[r] &(j_{f \circ \overline{g}} \circ \theta_{f \circ \overline{g}})_* u_{\eta\#}\arrow[r] & \SW(u,j_{f \circ \overline{g}} \circ \theta_{f \circ \overline{g}}) \arrow[r] & + 1.
    \end{tikzcd}
\end{equation*}
We apply $\overline{g}_{\sigma*}(p_{\Delta \times \mathbb{N}^{\times}})_{\#}(i_{f \circ \overline{g}})^*$ to the left and $(\theta_{f \circ g})^*(p_{\Delta \times \mathbb{N}^{\times}})^*$ to the right to get 
\begin{equation*}
    \begin{tikzcd}[sep=large]
        \square_1 \arrow[r] & \square_2 \arrow[r]  & \square_3 \arrow[r]  & + 1 
    \end{tikzcd}
\end{equation*}
where 
\begin{align*}
    \square_1 & = \overline{g}_{\sigma*}(p_{\Delta \times \mathbb{N}^{\times}})_{\#} {\color{blue}(i_{f \circ \overline{g}})^*u_{\#}}(j_{f \circ g} \circ \theta_{f \circ g})_{*}(\theta_{f \circ g})^*(p_{\Delta \times \mathbb{N}^{\times}})^* \\ 
    & \simeq \overline{g}_{\sigma*}{\color{blue}(p_{\Delta \times \mathbb{N}^{\times}})_{\#}u_{\sigma \#}}(i_{f \circ g})^*(j_{f \circ g} \circ \theta_{f \circ g})_{*} (\theta_{f \circ g})^*(p_{\Delta \times \mathbb{N}^{\times}})^* \\ 
    & = {\color{blue}\overline{g}_{\sigma*}u_{\sigma \#}}(p_{\Delta \times \mathbb{N}^{\times}})_{\#}(i_{f \circ g})^*(j_{f \circ g} \circ \theta_{f \circ g})_* (\theta_{f \circ g})^*(p_{\Delta \times \mathbb{N}^{\times}})^* \\ 
    & = g_{\sigma!}\Psi_{f \circ g}
\end{align*}
and 
\begin{align*}
    \square_2 & = \overline{g}_{\sigma*}(p_{\Delta \times \mathbb{N}^{\times}})_{\#}(i_{f \circ \overline{g}})^*(j_{f \circ \overline{g}} \circ \theta_{f \circ \overline{g}})_*{\color{blue}u_{\eta\#}(\theta_{f \circ g})^*(p_{\Delta \times \mathbb{N}^{\times}})^*} \\ 
    & \simeq {\color{blue}\overline{g}_{\sigma*}(p_{\Delta \times \mathbb{N}^{\times}})_{\#}}(i_{f \circ \overline{g}})^*(j_{f \circ \overline{g}}\circ \theta_{f \circ \overline{g}})_*(\theta_{f \circ \overline{g}})^*(p_{\Delta \times \mathbb{N}^{\times}})^* u_{\eta \#}\\ 
    & \simeq (p_{\Delta \times \mathbb{N}^{\times}})_{\#} {\color{blue}\overline{g}_{\sigma*}(i_{f \circ \overline{g}})^*}(j_{f \circ \overline{g}} \circ \theta_{f \circ \overline{g}})_*(\theta_{f \circ \overline{g}})^*(p_{\Delta \times \mathbb{N}^{\times}})^* u_{\eta \#} \\ 
    & \simeq (p_{\Delta \times \mathbb{N}^{\times}})_{\#} (i_f)^*{\color{blue}\overline{g}_*(j_{f \circ \overline{g}} \circ \theta_{f \circ \overline{g}})_*}(\theta_{f \circ \overline{g}})^*(p_{\Delta \times \mathbb{N}^{\times}})^* u_{\eta \#} \\ 
    & = (p_{\Delta \times \mathbb{N}^{\times}})_{\#} (i_f)^*(j_f \circ \theta_f)_* {\color{blue}\overline{g}_{\eta *}(\theta_{f \circ \overline{g}})^*(p_{\Delta \times \mathbb{N}^{\times}})^*} u_{\eta \#} \\ 
    & \simeq  (p_{\Delta \times \mathbb{N}^{\times}})_{\#} (i_f)^*(j_f \circ \theta_f)_* (\theta_f)^*(p_{\Delta \times \mathbb{N}^{\times}})^* {\color{blue}\overline{g}_{\eta *} u_{\eta \#}} \\ 
    & = \Psi_f g_{\eta !}.
\end{align*}
We define 
\begin{equation*}
\Psi^{\infty}_{f,g}  \coloneqq \square_3 = \overline{g}_{\sigma*}(p_{\Delta \times \mathbb{N}^{\times}})_{\#}(i_{f \circ \overline{g}})^*\SW(u,j_{f \circ \overline{g}} \circ \theta_{f \circ \overline{g}})(\theta_{f \circ g})^*(p_{\Delta \times \mathbb{N}^{\times}})^*
\end{equation*}
as desired. 
\end{proof}

\subsection{Uniqueness of motivic nearby cycles functors at infinity} So far, we have been working with functors $\Psi^{\infty}_{f,g}$ without proving its independency on compactifications. It turns out that in general, $\Psi^{\infty}_{f,g}$ may depend on the choice of the compactification, but when the domain of $g$ is smooth over $k$, we can prove that $\Psi^{\infty}_{f,g}$ is independent of the choice of the compactification. Our subsequent proofs follow the ideas in \cite{raibaut-2012} though the techniques are different. 
\begin{prop}[Direct image formula] \label{direct image formula}
    Let $X,Y$ be $k$-varieties and $g \colon Y \longrightarrow X, f \colon X \longrightarrow \mathbb{A}^1_k$ be morphisms of $k$-varieties. Assume that $Y$ is smooth over $k$. Given a compactification
    \begin{equation*}
           \begin{tikzcd}[sep=large]
        Y \arrow[r,"u"] \arrow[rd,"g",swap] & \overline{Y} \arrow[d,"\overline{g}"] \\ 
        & X.
    \end{tikzcd} 
    \end{equation*}
    of $g$ and let $v \colon F = (\overline{Y}\setminus Y)_{\mathrm{red}} \longhookrightarrow \overline{Y}$ be the (reduced) closed complement of $Y$. Let $h \colon (Z,E) \longrightarrow (\overline{Y},F)$ be a log-resolution of singularities and denote by $i$ the open immersion
    \begin{equation*}
        i \coloneqq h^{-1} \circ u  \colon Y \longhookrightarrow Z,
    \end{equation*}
    here we are making an abuse of notation by writing $h^{-1}$ for the inverse of the restriction of $h$ to $Z \setminus E$. Then there exists a natural isomorphism of triangulated functors 
   \begin{equation*}
        \overline{g}_{\sigma*}\Psi^{\infty}_{f \circ \overline{g},u} \simeq (\overline{g} \circ h)_{\sigma *}\Psi^{\infty}_{f \circ \overline{g} \circ h,i}.
    \end{equation*}
\end{prop}

\begin{proof}
    We form the cartesian square
    \begin{equation*}
        \begin{tikzcd}[sep=large]
          E \arrow[r,hook,"t"] \arrow[d,"h_{\mid E}",swap] & Z \arrow[d,"h"] \\ 
            F \arrow[r,"v",hook] & \overline{Y}. 
        \end{tikzcd}
    \end{equation*}
    Let $s\colon E_{\mathrm{red}} \longrightarrow E$ be the canonical immersion, then since $E \setminus E_{\mathrm{red}} = \varnothing$ (set theoretically), we see that $s_*$ is an equivalence of categories, i.e. $s_*s^* \simeq \id$. Note also that $\mathrm{BC}^h_{\eta} \colon u_{\eta \#} \overset{\sim}{\longrightarrow} h_{\eta*}i_{\eta \#}$ is an isomorphism. Thus we have
    \begin{align*}
         \overline{g}_{\sigma*}\Psi^{\infty}_{f \circ \overline{g},u} & =  \overline{g}_{\sigma*}(p_{\Delta \times \mathbb{N}^{\times}})_{\#}(i_{f \circ \overline{g}})^*v_{*}v^*(j_{f \circ \overline{g}} \circ \theta_{f \circ \overline{g}})_* {\color{blue}u_{\eta \#}}(\theta_{f \circ g})^*(p_{\Delta \times \mathbb{N}^{\times}})^* \\ 
         & \simeq \overline{g}_{\sigma*}(p_{\Delta \times \mathbb{N}^{\times}})_{\#}(i_{f \circ \overline{g}})^*v_{*}v^*{\color{blue}(j_{f \circ \overline{g}} \circ \theta_{f \circ \overline{g}})_* h_{\eta *}}i_{\eta \#}(\theta_{f \circ g})^*(p_{\Delta \times \mathbb{N}^{\times}})^* \\ 
         & = \overline{g}_{\sigma*}(p_{\Delta \times \mathbb{N}^{\times}})_{\#}(i_{f \circ \overline{g}})^*v_*{\color{blue}v^*h_*}(j_{f \circ \overline{g} \circ h} \circ \theta_{f \circ \overline{g}\circ h})_* i_{\eta \#}(\theta_{f \circ g})^*(p_{\Delta \times \mathbb{N}^{\times}})^* \\
         & \simeq \overline{g}_{\sigma*}(p_{\Delta \times \mathbb{N}^{\times}})_{\#}(i_{f \circ \overline{g}})^*{\color{blue}v_*(h_{\mid E})_*}t^*(j_{f \circ \overline{g} \circ h} \circ \theta_{f \circ \overline{g} \circ h})_* i_{\eta \#}(\theta_{f \circ g})^*(p_{\Delta \times \mathbb{N}^{\times}})^* \\ 
         & = \overline{g}_{\sigma*}(p_{\Delta \times \mathbb{N}^{\times}})_{\#}{\color{blue}(i_{f \circ \overline{g}})^*h_*}t_*t^*(j_{f \circ \overline{g} \circ h} \circ \theta_{f \circ \overline{g} \circ h})_* i_{\eta \#}(\theta_{f \circ g})^*(p_{\Delta \times \mathbb{N}^{\times}})^* \\ 
         & \simeq \overline{g}_{\sigma*}{\color{blue}(p_{\Delta \times \mathbb{N}^{\times}})_{\#}h_{\sigma*}}(i_{f \circ \overline{g} \circ h})^*t_*t^*(j_{f \circ \overline{g} \circ h} \circ \theta_{f \circ \overline{g} \circ h})_* i_{\eta \#}(\theta_{f \circ g})^*(p_{\Delta \times \mathbb{N}^{\times}})^* \\ 
         & \simeq \overline{g}_{\sigma*}h_{\sigma*}(p_{\Delta \times \mathbb{N}^{\times}})_{\#}(i_{f \circ \overline{g} \circ h})^*{\color{blue}t_*t^*}(j_{f \circ \overline{g} \circ h} \circ \theta_{f \circ \overline{g} \circ h})_* i_{\eta \#}(\theta_{f \circ g})^*(p_{\Delta \times \mathbb{N}^{\times}})^* \\ 
         &  \simeq \overline{g}_{\sigma*}h_{\sigma*}(p_{\Delta \times \mathbb{N}^{\times}})_{\#}(i_{f \circ \overline{g} \circ h})^*t_*s_*s^*t^*(j_{f \circ \overline{g} \circ h} \circ \theta_{f \circ \overline{g}\circ h})_* i_{\eta \#}(\theta_{f \circ g})^*(p_{\Delta \times \mathbb{N}^{\times}})^* \\ 
         & = (\overline{g} \circ h)_{\sigma *}\Psi^{\infty}_{f \circ \overline{g} \circ h,i}
    \end{align*}
    as desired.
\end{proof}

\begin{lem} \label{biresolution of singularities}
 Let $\overline{Y}_1,\overline{Y}_2,X$ be $k$-varieties. Let $\overline{g}_1 \colon \overline{Y}_1 \longrightarrow X, \overline{g}_2 \colon \overline{Y}_2 \longrightarrow X$ be proper morphisms of $k$-varieties. Let $v_1 \colon F_1 \longhookrightarrow \overline{Y}_1, v_2\colon F_2 \longrightarrow \overline{Y}_2$ be closed immersions so that there exists an isomorphism $\Phi \colon \overline{Y}_1 \setminus F_1 \overset{\sim}{\longrightarrow} \overline{Y}_2 \setminus F_2$. Assume furthermore that $\overline{Y}_1\setminus F_1$ is smooth over $k$. There exists a quadruplet $(Z,E,h_1,h_2)$ such that $h_1 \colon (Z,E) \longrightarrow (\overline{Y}_1,F_1)$, $h_2 \colon (Z,E) \longhookrightarrow (\overline{Y}_2,F_2)$ are log-resolutions and $\overline{g}_1 \circ h_1 = \overline{g}_2 \circ h_2$. 
\end{lem}
\begin{proof}
  The proof of \cite[Théorème 2.8]{raibaut-2012} carries over word by word to this situation.  
\end{proof}

\begin{theorem} \label{independency of nearby cycles functors at infinity}
    Let $X,Y$ be $k$-varieties with $Y$ smooth over $k$ and $g \colon Y \longrightarrow X, f \colon X \longrightarrow \mathbb{A}^1_k$ be morphisms of $k$-varieties. Given two compactifications 
    \begin{equation*}
           \begin{tikzcd}[sep=large]
        Y \arrow[r,"u_1"] \arrow[rd,"g",swap] & \overline{Y}_1 \arrow[d,"\overline{g}_1"] \\ 
        & X.
    \end{tikzcd} \ \ \ \ \ \ \ \ 
       \begin{tikzcd}[sep=large]
        Y \arrow[r,"u_2"] \arrow[rd,"g",swap] & \overline{Y}_2 \arrow[d,"\overline{g}_2"] \\ 
        & X,
    \end{tikzcd} 
    \end{equation*}
    then there exists a natural isomorphism of triangulated functors 
    \begin{equation*}
        \hat{g}_{1\sigma*}\Psi^{\infty}_{f \circ \overline{g}_1,u_1} \simeq \hat{g}_{2\sigma*}\Psi^{\infty}_{f \circ \overline{g}_2,u_2}.
    \end{equation*}
    In other words, the triangulated functor $\Psi^{\infty}_{f,g}$ is well-defined, i.e. independent of the choice of the compactification, whenever $Y$ is smooth over $k$.
\end{theorem}
\begin{proof}
    By lemma \ref{biresolution of singularities}, we have a commutative diagram
      \begin{equation*}
        \begin{tikzcd}[sep=large]
            & (Z,E) \arrow[dr,"h_2"] \arrow[dl,"h_1",swap] & \\ 
            \overline{Y}_1 \arrow[ddr,"\overline{g}_1",swap] & Y \arrow[dd,"g"] \arrow[r,"u_2"] \arrow[l,"u_1",swap] & \overline{Y}_2 \arrow[ddl,"\overline{g}_2"] \\ 
            &   &  \\ 
            & X & 
        \end{tikzcd}
    \end{equation*}
    in which $h_1\colon (Z,E) \longrightarrow (\overline{Y}_1,(\overline{Y}_1\setminus Y)_{\mathrm{red}})$ and $h_2\colon (Z,E) \longrightarrow (\overline{Y}_2,(\overline{Y}_2 \setminus Y)_{\mathrm{red}})$ are log-resolutions. Denote by $i$ the open immersion
    \begin{equation*}
        i \coloneqq h_1^{-1} \circ u_1 = h_2^{-1} \circ u_2 \colon Y \longhookrightarrow Z,
    \end{equation*}
    here we are making an abuse of notation, i.e. $h_1^{-1},h_2^{-1}$ are inverses of restrictions of $h_1,h_2$ to $Z \setminus E$, respectively. We then have 
    \begin{equation*}
        \hat{g}_{1\sigma*}\Psi^{\infty}_{f \circ \overline{g}_1,u_1} \simeq (\hat{g}_{1} \circ h_1)_{\sigma *}\Psi^{\infty}_{f \circ \overline{g}_1 \circ h_1,i} = (\hat{g}_{2} \circ h_2)_{\sigma *}\Psi^{\infty}_{f \circ \overline{g}_2 \circ h_2,i} \simeq  \hat{g}_{2\sigma*}\Psi^{\infty}_{f \circ \overline{g}_2,u_2}
    \end{equation*}
    by the direct image formula \ref{direct image formula} and the commutativity of the diagram above.
\end{proof}

\subsection{The category of quasi-unipotent motives} Let $S$ be a scheme, the category $\qush(S)$ of \textit{quasi-unipotent motives} is a full triangulated subcategory of $\sh(\gm_{m,S})$. To be more explicit, for every $S$-scheme $Y$, $g \in \mathcal{O}_Y(Y)^{\times}$, $n \in \mathbb{N}^{\times}$, we define the $\gm_{m,S}$-scheme 
\begin{equation*}
    Q^{\textnormal{gm}}_n(Y,g) \coloneqq \Spec\left(\frac{\mathcal{O}_Y[T,T^{-1},V]}{(V^n - gT)} \right) \longrightarrow \Spec(\mathcal{O}_S[T,T^{-1}]) = \gm_{m,S}.
\end{equation*}
The category $\qush(S)$ is the smallest full triangulated subcategory of $\sh(\gm_{m,S})$ stable under all small direct sums and containing the objects of the form $\pi_{\#}\mathds{1}_{Q^{\textnormal{gm}}_n(Y,g)}(r)$ where $Y$ is a smooth $S$-scheme of finite type, $g$ in $\mathcal{O}_Y(Y)^{\times}$, $n$ in $\mathbb{N}^{\times}$ and $r$ in $\mathbb{Z}$ and $\pi\colon Q^{\textnormal{gm}}_n(Y,g) \longrightarrow \gm_{m,S}$ is the structural morphism. The category of quasi-unipotent motives, originally introduced in \cite{ayoub-2015} for the case $S$ being a field, and then is generalized for arbitrary $S$ and studied intensively in \cite{florian+julien-2021}. If $f\colon X \longrightarrow Y$ is a morphism of $k$-varieties, then $\breve{f}^*$ preserves quasi-unipotent motives by \cite[Lemma 3.2.1]{florian+julien-2021}. If $f$ is smooth, then $\breve{f}_{\#}$ preserves quasi-unipotent motives by \cite[Lemma 3.2.1]{florian+julien-2021} as well. If moreover, $X,Y$ are quasi-projective varieties over $k$, then $\breve{f}_!$ preserves quasi-unipotent motives by \cite[Proposition 3.2.3]{florian+julien-2021}. In the following proposition, we remove this assumption on quasi-projectiveness. 

\begin{prop} \label{exceptional direct images preserve quasi-unipotent motives}
    Let $X,Y$ be $k$-varieties and $f \colon X \longrightarrow Y$ be a morphism of $k$-varieties. If $A \in \qush(X)$, then $\breve{f}_!(A)$ is in $\qush(Y)$. 
\end{prop}
\begin{proof}
    By the Nagata compactification theorem and \cite[Lemma 3.2.1]{florian+julien-2021}, we may reduce to the case that $f$ is proper which implies that $\breve{f}_! = \breve{f}_*$. If $f$ is of the form $\mathbb{P}^n_Y \longrightarrow Y$, which is smooth, then the proposition is a special case of the proof of \cite[Proposition 3.2.3]{florian+julien-2021}. If $f$ is a closed immersion, then the proposition follows from \cite[Lemma 3.2.2]{florian+julien-2021}. Consequently, the proposition holds true for every projective morphism. The general case is handled by noetherian induction and the Chow lemma. First we prove the following property of closed subschemes of $X$ holds true: \medskip \\ 
    \begin{itshape} 
    Let $t \colon T \longhookrightarrow X$ be a proper closed subscheme, if for any proper closed subscheme $z \colon Z \longhookrightarrow T$, $\breve{f}_*\breve{(t \circ z)}_*\breve{(t \circ z)}^*$ preserves quasi-unipotent motives, then $\breve{f}_*\breve{t}_*\breve{t}^*$ preserves quasi-unipotent motives. \medskip \\ 
    \end{itshape}
    Since $f \circ t$ is proper, by the Chow lemma, there exists a projective morphism $p \colon V \longrightarrow T$ such that $f \circ t \circ p$ is also projective and there exists an open dense subscheme $u \colon U \longhookrightarrow T$ of $T$ over which $p$ is an isomorphism. Let $z\colon Z \longhookrightarrow T$ be the (reduced) closed complement of $U$ in $T$. We have the distinguished triangle
    \begin{equation*}
        \breve{u}_{\#}\breve{u}^* \longrightarrow \id \longrightarrow \breve{z}_*\breve{z}^* \longrightarrow +1.
    \end{equation*}
    Apply $\breve{f}_*\breve{t}_*$ from the left and $\breve{t}^*$ from the right, we obtain a new distinguished triangle
    \begin{equation*}
        \breve{f}_*\breve{t}_*\breve{u}_{\#}\breve{u}^*\breve{t}^* \longrightarrow \breve{f}_*\breve{t}_*\breve{t}^* \longrightarrow \breve{f}_*\breve{t}_*\breve{z}_*\breve{z}^*\breve{t}^* \longrightarrow +1.
    \end{equation*}
    By the induction hypothesis, we know that $\breve{f}_*\breve{(t \circ z)}_*\breve{(t \circ z)}^*$ preserves quasi-unipotent motives so it suffices to prove that $\breve{f}_*\breve{t}_*\breve{u}_{\#}\breve{u}^*\breve{t}^*$ preserves quasi-unipotent motives. By \cite[Lemma 3.2.1]{florian+julien-2021}, we reduce to proving that $\breve{(f \circ t)}_*\breve{u}_{\#}$ preserves quasi-unipotent motives. We form the cartesian square
    \begin{equation*}
         \begin{tikzcd}[sep=large]
          p^{-1}(U) \arrow[r,"j"] \arrow[d,"q",swap] & V \arrow[d,"p"] \\ 
            U \arrow[r,"u"] & T.
        \end{tikzcd}
        \end{equation*}
        Since $q$ is an isomorphism, we deduce that $\breve{u}_{\#}= \breve{p}_*\breve{j}_{\#}\breve{q}^!$. Consequently, we obtain $\breve{(f \circ t)}_*u_{\#} = \breve{(f \circ t \circ p)}_*\breve{j}_{\#}\breve{q}^!$. By the projective case and \cite[Lemma 3.2.1]{florian+julien-2021} again, we reduce the problem to proving that $\breve{q}^!$ preserves quasi-unipotent motives, but $q$ is an isomorphism so we are done. We finish by noetherian induction. We define
        \begin{equation*}
            \mathscr{S} = \left \{t\colon T \longhookrightarrow X \mid T \ \text{nonempty, proper closed s.t} \ \breve{f}_*\breve{t}_*\breve{t}^*(\qush(X)) \not\subset \qush(Y) \right \}.  
        \end{equation*}
        If $\mathscr{S} \neq \varnothing$, then we can find an infinite descending chain in $\mathscr{S}$ as follows: for any $T \in \mathscr{S}$, there exists a proper closed subscheme $Z$ of $T$ such that $Z \in \mathscr{S}$. But this contradicts the noetherian property of $X$. Consequently, $\mathscr{S} = \varnothing$, which implies that $\breve{f}_!$ preserves quasi-unipotent motives. 
\end{proof}
\subsection{The monodromic nearby cycles functors at infinity} As a must-have consequence of the previous section, the \textit{monodromic nearby cycles functor} constructed in \cite{florian+julien-2021} has to admit its monodromic nearby cycles functors at infinity. Let $X$ be a $k$-variety and $f \colon X \longrightarrow \mathbb{A}_k^1$ be a morphism of $k$-varieties. We form the diagram
\begin{equation*} \label{monodromic nearby cycles}
    \begin{tikzcd}[sep=large]
      \gm_{m,X_{\eta}} \arrow[d,"f_{\eta}^{\gm_m}",swap] \arrow[r]  & \gm_{m,X} \arrow[d,"f^{\gm_m}"] & \gm_{m,X_{\sigma}} \arrow[l] \arrow[d,"f_{\sigma}^{\gm_m}"] \\ 
     \gm_{m,k} \arrow[r,"j"] & \mathbb{A}_k^1 & \mathrm{Spec}(k) \arrow[l,"i",swap]
    \end{tikzcd}
\end{equation*}
where we define $f^{\gm_m}\colon \Spec(\mathcal{O}_X[T,T^{-1}]) \longrightarrow \mathbb{A}^1_k = \Spec(k[t])$ to be the morphism given by $t \longmapsto fT$ once we identify $f$ with the image of $t$ under the canonical morphism $k[t] \longrightarrow \mathcal{O}_X(X)$. With all these data, the monodromic nearby cycles functor is defined as
\begin{equation*}
    \Psi_f^{\textnormal{mon}} \coloneqq \Psi_{f^{\gm_m}}  p^*\colon \sh(X_{\eta}) \longmapsto \sh(\gm_{m,X_{\sigma}}).
\end{equation*}
in which $p\colon \gm_{m,X_{\eta}} \longrightarrow X_{\eta}$ is the canonical projection. It was proved in \cite{florian+julien-2021} that $\Psi^{\textnormal{mon}}_f$ takes values in $\qush(X_{\sigma})$ (see \cite[Theorem 4.2.1]{florian+julien-2021}). We can recover $\Psi_f$ from $\Psi^{\textnormal{mon}}_f$ thanks to \cite[Theorem 4.1.1]{florian+julien-2021}. Here we reprove \cite[Theorem 4.1.1]{florian+julien-2021} by using the explicit definition of nearby cycles functors and therefore avoid using resolutions of singularities.
\begin{prop} \label{IS21, theorem 4.1.1}
    Let $X$ be a $k$-variety and $f \colon X \longrightarrow \mathbb{A}_k^1$ be a morphism of $k$-varieties, then there is a natural isomorphism
    \begin{equation*}
        1^*\Psi_f^{\textnormal{mon}} \overset{\sim}{\longrightarrow} \Psi_f.
    \end{equation*}
\end{prop}
\begin{proof}
By successively applying proposition \ref{proper base change for derivators} to cartesian diagrams
\begin{equation*}
    \begin{tikzcd}[sep=large]
        \gm_{m,X_{\sigma}} \arrow[d,"i_{f^{\gm_m}}",swap] & X_{\sigma} \arrow[l,"1",swap] \arrow[d,"i_f"] \\ 
        \gm_{m,X} & X \arrow[l,"1",swap] 
    \end{tikzcd} \ \ \ \ \ \ \ \ \ \ \ \ \ \ \  \begin{tikzcd}[row sep=large, column sep = huge]
       (\mathscr{R}_{\gm_{m,X}},\Delta \times \mathbb{N}^{\times}) \arrow[r,"j_{f^{\gm_m}} \circ \theta_{f^{\gm_m}}"] \arrow[d,"p",swap] & (\gm_{m,X},\Delta \times \mathbb{N}^{\times})  \arrow[d,"p"]  \\ 
       (\mathscr{R}_X,\Delta \times \mathbb{N}^{\times}) \arrow[r,"j_f \circ \theta_f"] & (X,\Delta \times \mathbb{N}^{\times}),
    \end{tikzcd}
\end{equation*}
we see that
    \begin{align*}
        1^*\Psi_{f^{\gm_m}}p^* & = {\color{blue}1^*(p_{\Delta \times \mathbb{N}^{\times}})_{\#}}(i_{f^{\gm_m}})^*(j_{f^{\gm_m}} \circ \theta_{f^{\gm_m}})_*(\theta_{f^{\gm_m}})^*(p_{\Delta \times \mathbb{N}^{\times}})^*p^* \\ 
        & \simeq (p_{\Delta \times \mathbb{N}^{\times}})_{\#}{\color{blue}1^*(i_{f^{\gm_m}})^*}(j_{f^{\gm_m}} \circ \theta_{f^{\gm_m}})_*(\theta_{f^{\gm_m}})^*(p_{\Delta \times \mathbb{N}^{\times}})^*p^* \\ 
        & = (p_{\Delta \times \mathbb{N}^{\times}})_{\#}(i_f)^* 1^*(j_{f^{\gm_m}} \circ \theta_{f^{\gm_m}})_*{\color{blue}(\theta_{f^{\gm_m}})^*(p_{\Delta \times \mathbb{N}^{\times}})^*p^*} \\ 
        & = (p_{\Delta \times \mathbb{N}^{\times}})_{\#}(i_f)^* 1^*{\color{blue}(j_{f^{\gm_m}} \circ \theta_{f^{\gm_m}})_*p^*}(\theta_f)^*(p_{\Delta \times \mathbb{N}^{\times}})^*\\
        & \simeq (p_{\Delta \times \mathbb{N}^{\times}})_{\#}(i_f)^*{\color{blue} 1^*p^*}(j_f \circ \theta_f)_*(\theta_f)^*(p_{\Delta \times \mathbb{N}^{\times}})^*\\
        & = \Psi_f
    \end{align*}
    as desired.
\end{proof}
\begin{lem} \label{proper base change theorem for monodromic nearby cycles functor} 
     Let $X,Y$ be $k$-varieties and $g\colon Y \longrightarrow X, f\colon X \longrightarrow \mathbb{A}_k^1$ be morphisms of $k$-varieties, then there exists a natural transformation 
      \begin{equation*}
         \begin{tikzcd}[sep=large]
         \sh(Y_{\eta}) \arrow[r,"\Psi^{\textnormal{mon}}_{f \circ g}"] \arrow[d,"g_{\eta!}",swap] & \qush(Y_{\sigma}) \arrow[Rightarrow, shorten >=30pt, shorten <=30pt, dl,"\mu_g^{\textnormal{mon}}"] \arrow[d,"\breve{g}_{\sigma!}"] \\ 
         \sh(X_{\eta}) \arrow[r,"\Psi^{\textnormal{mon}}_{f}",swap] &  \qush(X_{\sigma}).
    \end{tikzcd} 
    \end{equation*}
   which is an isomorphism if $g$ is proper. If $h \colon Z \longrightarrow Y$ is another morphism of $k$-varieties, then we have a commutative diagram 
   \begin{equation*}
       \begin{tikzcd}[sep=large]
               \breve{(g \circ h)}_{\sigma!}\Psi_{f \circ g \circ h}^{\textnormal{mon}}\arrow[d,"\mu_h^{\textnormal{mon}}",swap] \arrow[r,"\mu_{g \circ h}^{\textnormal{mon}}"] &  \Psi_{f}^{\textnormal{mon}}(g \circ h)_{\eta!}\arrow[d,"\sim"]    \\ 
              \breve{g}_{\sigma!}\Psi_{f \circ g}^{\textnormal{mon}}h_{\eta !} \arrow[r,"\mu_g^{\textnormal{mon}}"] &  \Psi_{f}^{\textnormal{mon}}(g \circ h)_{\eta !}.
       \end{tikzcd}
   \end{equation*}
\end{lem}
\begin{proof}
    Consider two diagrams below, we have an isomorphism in the left diagram because of the proper base change theorem
    \begin{equation*}
         \begin{tikzcd}[sep=large]
            \sh(Y_{\eta}) \arrow[r,"p^*"] \arrow[d,"g_{\eta!}",swap] & \sh(\gm_{m,Y_{\eta}}) \arrow[d,"\breve{g}_{\eta!}"] \arrow[Rightarrow, shorten >=29pt, shorten <=29pt, dl,"\Ex_!^*"]  \\
            \sh(X_{\eta}) \arrow[r,"p^*"] & \sh(\gm_{m,X_{\eta}})
        \end{tikzcd} \ \ \ \ \ \ \ \  \begin{tikzcd}[sep=large]
        \sh(\gm_{m,Y_{\eta}})\arrow[r,"\Psi_{(f\circ g)^{\gm_m}}"] \arrow[d,"\breve{g}_{\eta !}",swap] & \sh(\gm_{m,Y_{\sigma}}) \arrow[Rightarrow, shorten >=33pt, shorten <=33pt, dl,"\mu_{\breve{g}}"] \arrow[d,"\breve{g}_{\sigma!}"] \\ 
         \sh(\gm_{m,X_{\eta}})  \arrow[r,"\Psi_{f^{\gm_m}}"] &  \sh(\gm_{m,X_{\sigma}}).
    \end{tikzcd}
    \end{equation*}
    Since the compositions $\Psi_{f^{\gm_m}}p^*$ and $\Psi_{(f\circ g)^{\gm_m}}p^*$ take values in $\qush$, we can match these two diagrams to obtain the desired natural transformation. If furthermore, $g$ is proper then base change of the right diagram is an isomorphism by definition. The compatibility of $\mu_?^{\textnormal{mon}}$ with respect to composition follows from the compatibility of $\mu_?$ (see proposition \ref{compatibility of nearby cycles functors with exceptional direct images}) and of the base change structure $\Ex_!^*$. 
    \end{proof} 

    \begin{cor}
        Let $X,Y$ be $k$-varieties and $g \colon Y \longrightarrow X$, $f\colon X \longrightarrow \mathbb{A}_k^1$ be morphisms of $k$-varieties. There exists a triangulated functor 
        \begin{equation*}
            \Psi^{\textnormal{mon},\infty}_{f,g} \colon \sh(Y_{\eta}) \longrightarrow \qush(X_{\sigma})
        \end{equation*}
        fitting into a distinguished triangle
        \begin{equation*}
               \breve{g}_{\sigma!}\Psi_{f \circ g}^{\textnormal{mon}}\overset{\mu_g^{\textnormal{mon}}}{\longrightarrow}\Psi_{f}^{\textnormal{mon}}g_{\eta !} \longrightarrow \Psi^{\textnormal{mon},\infty}_{f,g} \longrightarrow +1.
        \end{equation*} 
        Moreover, given a compactification
    \begin{equation*}
           \begin{tikzcd}[sep=large]
        Y \arrow[r,"u"] \arrow[rd,"g",swap] & \overline{Y} \arrow[d,"\overline{g}"] \\ 
        & X.
    \end{tikzcd} 
    \end{equation*}
    of $g$ and let $v \colon F = (\overline{Y}\setminus Y)_{\mathrm{red}} \longhookrightarrow \overline{Y}$ be the (reduced) closed complement of $Y$. Let $h \colon (Z,E) \longrightarrow (\overline{Y},F)$ be a log-resolution of singularities and denote by $i$ the open immersion
    \begin{equation*}
        i \coloneqq h^{-1} \circ u  \colon Y \longhookrightarrow Z,
    \end{equation*}
    here we are making an abuse of notation by writing $h^{-1}$ for the inverse of the restriction of $h$ to $Z \setminus E$. Then there exists a natural isomorphism of triangulated functors 
   \begin{equation*}
        \breve{\overline{g}}_{\sigma*}\Psi^{\textnormal{mon},\infty}_{f \circ \overline{g},u} \simeq \breve{(\overline{g} \circ h)}_{\sigma *}\Psi^{\textnormal{mon},\infty}_{f \circ \overline{g} \circ h,i}.
    \end{equation*}
    \end{cor}
\begin{proof}
 We start with the distinguished triangle 
 \begin{equation*}
     \breve{g}_{\sigma!}\Psi_{(f \circ g)^{\gm_m}} \longrightarrow \Psi_{f^{\gm_m}}\breve{g}_{\eta !} \longrightarrow \Psi^{\infty}_{f^{\gm_m},\breve{g}} \longrightarrow + 1
 \end{equation*}
 We apply $p^*$ and naturally, we define $\Psi_{f,g}^{\textnormal{mon},\infty} \coloneqq \Psi^{\infty}_{f^{\gm_m},\breve{g}}p^* \colon \sh(Y_{\eta}) \longrightarrow \sh(\gm_{m,X_{\sigma}})$. The fact that the images of $\Psi^{\textnormal{mon},\infty}_{f,g}$ lie in $\qush(X_{\sigma})$ follows from proposition \ref{exceptional direct images preserve quasi-unipotent motives}. Theorem \ref{independency of nearby cycles functors at infinity} implies the second part. 
\end{proof}

The proposition below is an analogue at infinity of \cite[Theorem 4.1.1]{florian+julien-2021} and proposition \ref{IS21, theorem 4.1.1}. It tells us that we can recover nearby cycles functors at infinity from their monodromic counterparts. 
\begin{prop}
   Let $X,Y$ be $k$-varieties and $g \colon Y \longrightarrow X$, $f \colon X \longrightarrow \mathbb{A}_k^1$ be morphisms of $k$-varieties. There is a natural isomorphism 
    \begin{equation*}
          1^*\Psi^{\textnormal{mon},\infty}_{f,g} \overset{\sim}{\longrightarrow} \Psi^{\infty}_{f,g}.
    \end{equation*}
\end{prop}
\begin{proof}
    Let 
    \begin{equation*}
           \begin{tikzcd}[sep=large]
        Y \arrow[r,"u"] \arrow[rd,"g",swap] & \overline{Y} \arrow[d,"\overline{g}"] \\ 
        & X.
    \end{tikzcd} 
    \end{equation*} 
    be a compactification of $g$. Let $v \colon F \longhookrightarrow \overline{Y}$ the the (reduced) closed complement of $u$. There are cartesian squares
\begin{equation*}
\begin{tikzcd}[sep=large]
        \gm_{m,\overline{Y}_{\sigma}} \arrow[d,"\breve{\overline{g}}_{\sigma}",swap] & \overline{Y} \arrow[l,"1",swap] \arrow[d,"\overline{g}_{\sigma}"] \\ 
        \gm_{m,X_{\sigma}} & X_{\sigma}\arrow[l,"1",swap] 
    \end{tikzcd} \qquad \begin{tikzcd}[sep=large]
  \gm_{m,F} \arrow[d,"\breve{v}",swap]   & \gm_{m,F_{\sigma}} \arrow[d,"\breve{v}_{\sigma}"]  \arrow[l,swap,"i_{f^{\gm_m} \circ \breve{\overline{g}} \circ \breve{v}}"] &  F_{\sigma} \arrow[l,"1",swap] \arrow[d,"v_{\sigma}"] \\ 
   \gm_{m,\overline{Y}}  & \gm_{m,\overline{Y}_{\sigma}}  \arrow[l,swap,"i_{f^{\gm_m} \circ \breve{\overline{g}}}"] & \overline{Y}_{\sigma} \arrow[l,"1",swap] 
    \end{tikzcd} 
\end{equation*}
\begin{equation*}
\begin{tikzcd}[sep=large]
      (\gm_{m,Y_{\eta}},\Delta \times \mathbb{N}^{\times}) \arrow[r,"\breve{u}_{\eta}"] \arrow[d,"p",swap] & (\gm_{m,\overline{Y}},\Delta \times \mathbb{N}^{\times}) \arrow[d,"p"] \\ 
       (Y_{\eta},\Delta \times \mathbb{N}^{\times}) \arrow[r,"u_{\eta}"] & (\overline{Y}_{\eta},\Delta \times \mathbb{N}^{\times})
   \end{tikzcd} \qquad
     \begin{tikzcd}[row sep=large, column sep = huge]
       (\mathscr{R}_{\gm_{m,\overline{Y}}},\Delta \times \mathbb{N}^{\times}) \arrow[r,"j_{f^{\gm_m}\circ \breve{\overline{g}}} \circ \theta_{f^{\gm_m}\circ \breve{\overline{g}}}"] \arrow[d,"p",swap] & (\gm_{m,\overline{Y}_{\eta}},\Delta \times \mathbb{N}^{\times})  \arrow[d,"p"]  \\ 
       (\mathscr{R}_{\overline{Y}},\Delta \times \mathbb{N}^{\times}) \arrow[r,"j_{f \circ \overline{g}}\circ \theta_{f\circ \overline{g}}"] & (\overline{Y}_{\eta},\Delta \times \mathbb{N}^{\times})
    \end{tikzcd}   
\end{equation*}
\begin{equation*}
    \begin{tikzcd}[sep=large]
        \gm_{m,\overline{Y}} \arrow[d,"p",swap] & \gm_{m,\overline{Y}_{\sigma}} \arrow[d,"p"] \arrow[l,"i_{f^{\gm_m} \circ \breve{\overline{g}}}",swap] & \overline{Y}_{\sigma} \arrow[ld,equal] \arrow[l,"1",swap] \\ 
        \overline{Y} & \overline{Y}_{\sigma} \arrow[l,"i_{f \circ \overline{g}}",swap] & 
    \end{tikzcd} \qquad \begin{tikzcd}[sep=large]
    F_{\sigma} \arrow[d,"v_{\sigma}",swap] \arrow[r,"i_{f \circ \overline{g} \circ v}"] & F \arrow[d,"v"] \\ 
    \overline{Y}_{\sigma} \arrow[r,"i_{f \circ \overline{g}}"] & \overline{Y}.
\end{tikzcd}
\end{equation*}
By successively applying proposition \ref{proper base change for derivators} to these diagrams, we see that
     \begin{align*}
        1^*\Psi^{\infty}_{f^{\gm_m},\breve{g}}p^* &  = {\color{blue}1^*\breve{\overline{g}}_{\sigma*}}(p_{\Delta \times \mathbb{N}^{\times}})_{\#}(i_{f^{\gm_m} \circ \breve{\overline{g}}})^*\breve{v}_*\breve{v}^*(j_{f^{\gm_m} \circ \breve{\overline{g}}} \circ \theta_{f^{\gm_m} \circ \breve{\overline{g}}})_* \breve{u}_{\eta \#}(\theta_{f^{\gm_m}\circ \breve{g}})^*(p_{\Delta \times \mathbb{N}^{\times}})^*p^*  \\ 
        & \simeq \overline{g}_{\sigma*}{\color{blue}1^*(p_{\Delta \times \mathbb{N}^{\times}})_{\#}}(i_{f^{\gm_m} \circ \breve{\overline{g}}})^*\breve{v}_*\breve{v}^*(j_{f^{\gm_m} \circ \breve{\overline{g}}} \circ \theta_{f^{\gm_m} \circ \breve{\overline{g}}})_* \breve{u}_{\eta \#}(\theta_{f^{\gm_m}\circ \breve{g}})^*(p_{\Delta \times \mathbb{N}^{\times}})^*p^* \\ 
        & \simeq \overline{g}_{\sigma*}(p_{\Delta \times \mathbb{N}^{\times}})_{\#}{\color{blue}1^*(i_{f^{\gm_m} \circ \breve{\overline{g}}})^*\breve{v}_*}\breve{v}^*(j_{f^{\gm_m} \circ \breve{\overline{g}}} \circ \theta_{f^{\gm_m} \circ \breve{\overline{g}}})_* \breve{u}_{\eta \#}(\theta_{f^{\gm_m}\circ \breve{g}})^*(p_{\Delta \times \mathbb{N}^{\times}})^*p^* \\ 
        & \simeq \overline{g}_{\sigma*}(p_{\Delta \times \mathbb{N}^{\times}})_{\#}v_{\sigma*}{\color{blue}1^*(i_{f^{\gm_m} \circ \breve{\overline{g}} \circ \breve{v}})^*\breve{v}^*}(j_{f^{\gm_m} \circ \breve{\overline{g}}} \circ \theta_{f^{\gm_m} \circ \breve{\overline{g}}})_* \breve{u}_{\eta \#}(\theta_{f^{\gm_m}\circ \breve{g}})^*(p_{\Delta \times \mathbb{N}^{\times}})^*p^* \\ 
        & = \overline{g}_{\sigma*}(p_{\Delta \times \mathbb{N}^{\times}})_{\#}v_{\sigma*}v_{\sigma}^*1^*(i_{f^{\gm_m} \circ \breve{\overline{g}}})^*(j_{f^{\gm_m} \circ \breve{\overline{g}}} \circ \theta_{f^{\gm_m} \circ \breve{\overline{g}}})_* \breve{u}_{\eta \#}{\color{blue}(\theta_{f^{\gm_m}\circ \breve{g}})^*(p_{\Delta \times \mathbb{N}^{\times}})^*p^*} \\ 
        & \simeq \overline{g}_{\sigma*}(p_{\Delta \times \mathbb{N}^{\times}})_{\#}v_{\sigma*}v_{\sigma}^*1^*(i_{f^{\gm_m} \circ \breve{\overline{g}}})^*(j_{f^{\gm_m} \circ \breve{\overline{g}}} \circ \theta_{f^{\gm_m} \circ \breve{\overline{g}}})_* {\color{blue}\breve{u}_{\eta \#}p^*}(\theta_{f\circ g})^*(p_{\Delta \times \mathbb{N}^{\times}})^* \\ 
        & \simeq \overline{g}_{\sigma*}(p_{\Delta \times \mathbb{N}^{\times}})_{\#}v_{\sigma*}v_{\sigma}^*1^*(i_{f^{\gm_m} \circ \breve{\overline{g}}})^*{\color{blue}(j_{f^{\gm_m} \circ \breve{\overline{g}}} \circ \theta_{f^{\gm_m} \circ \breve{\overline{g}}})_* p^*}u_{\eta \#}(\theta_{f\circ g})^*(p_{\Delta \times \mathbb{N}^{\times}})^* \\ 
        & \simeq \overline{g}_{\sigma*}(p_{\Delta \times \mathbb{N}^{\times}})_{\#}v_{\sigma*}v_{\sigma}^*{\color{blue}1^*(i_{f^{\gm_m} \circ \breve{\overline{g}}})^*p^*}(j_{f \circ \overline{g}} \circ \theta_{f \circ \overline{g}})_* u_{\eta \#}(\theta_{f\circ g})^*(p_{\Delta \times \mathbb{N}^{\times}})^* \\
        & =  \overline{g}_{\sigma*}(p_{\Delta \times \mathbb{N}^{\times}})_{\#}v_{\sigma*}{\color{blue}v_{\sigma}^*(i_{f \circ \overline{g}})^*}(j_{f \circ \overline{g}} \circ \theta_{f \circ \overline{g}})_* u_{\eta \#}(\theta_{f\circ g})^*(p_{\Delta \times \mathbb{N}^{\times}})^*  \\ 
        & = \overline{g}_{\sigma*}(p_{\Delta \times \mathbb{N}^{\times}})_{\#}{\color{blue}v_{\sigma*}(i_{f \circ \overline{g} \circ v})^*}v^*(j_{f \circ \overline{g}} \circ \theta_{f \circ \overline{g}})_* u_{\eta \#}(\theta_{f\circ g})^*(p_{\Delta \times \mathbb{N}^{\times}})^*  \\ 
        & \simeq \overline{g}_{\sigma*}(p_{\Delta \times \mathbb{N}^{\times}})_{\#}(i_{f \circ \overline{g}})^*v_*v^*(j_{f \circ \overline{g}} \circ \theta_{f \circ \overline{g}})_* u_{\eta \#}(\theta_{f \circ g})^*(p_{\Delta \times \mathbb{N}^{\times}})^* \\
         & = \Psi^{\infty}_{f,g}
    \end{align*}
    as desired.
\end{proof}

\begin{rmk}
    For technical reasons, we define the \textit{global} versions of $\Psi^{\textnormal{mon}}_f$ and $\Psi^{\textnormal{mon},\infty}_{f,g}$ by
    \begin{align*}
        \Psi^{\textnormal{mon}}_f (j_f)^* \colon & \sh(X) \longrightarrow \qush(X_{\sigma}) \\ 
        \Psi^{\textnormal{mon},\infty}_{f,g}(j_{f \circ g})^* \colon & \sh(Y) \longrightarrow \qush(X_{\sigma}).
    \end{align*}
    Since $(j_f)_*,(j_{f \circ g})_*$ are fully faithful, we can recover the functors $\Psi^{\textnormal{mon}}_f$ and $\Psi^{\textnormal{mon},\infty}_{f,g}$ from the functors $\Psi^{\textnormal{mon}}_f(j_f)^*$ and $\Psi^{\textnormal{mon},\infty}_{f,g}(j_{f \circ g})^*$ and therefore, from now on, we use these global functors without changing the notation. 
\end{rmk}

\section{Motivic nearby cycles at infinity in the the virtual world}

\subsection{Grothendieck rings and nearby cycles}
Given a $k$-variety $X$, the group $K_0(\var_X)$ is generated by isomorphisms classes of morphisms of $X$-varieties $p: U \longrightarrow X$ modulo the usual cut-paste relation. We denote by $\mathscr{M}_X$ the localization $K_0(\var_X)[\mathbf{L}^{-1}]$ with $\mathbf{L}=[\mathbb{A}_X^1]$ being the class of the affine line, called \textit{Grothendieck ring of $X$-varieties}. For a $k$-morphism $f \colon X \longrightarrow Y$, there are a group morphism and a ring morphism 
\begin{align*}
    f_! \colon \mathscr{M}_X & \longrightarrow \mathscr{M}_Y \\ 
    f^* \colon \mathscr{M}_Y & \longrightarrow \mathscr{M}_X
\end{align*}
given by composing and pulling back, respectively. Follow \cite{florian+julien-2013}, we can construct an Euler characteristic ring morphism 
\begin{align*}
    \chi_X \colon \mathscr{M}_X & \longrightarrow K_0(\shct(X)) \\ 
    [U] & \longmapsto [p_!(\mathds{1}_U)]
\end{align*}
compatible with proper push forwards $f_!$ and inverse images $f^*$. Similar, one can construct the group $\mathscr{M}^{\gm_m}_{X \times_k \gm_{m,k}}$, the Grothendieck ring of $\gm_{m,X}$-varieties $p \colon U \longrightarrow \gm_{m,X}$ endowed with a diagonally monomial $\gm_{m}$-action of some weight together with an Euler characteristic ring morphism 
\begin{align*}
    \chi^{\gm_{m}}_{X,c}\colon \mathscr{M}_{X \times_k \gm_{m,k}}^{\gm_{m}}  & \longrightarrow K_0(\qushct(X)) \\ 
    [U] & \longmapsto [p_!(\mathds{1}_U)]
    \end{align*} 
compatible with proper push forwards and inverse images 
\begin{align*}
    \breve{f}_! \colon \mathscr{M}^{\gm_m}_{X\times_k \gm_{m,k}} & \longrightarrow \mathscr{M}^{\gm_m}_{Y \times_k \gm_{m,k}} \\ 
    \breve{f}^* \colon \mathscr{M}^{\gm_m}_{Y \times_k \gm_{m,k}} & \longrightarrow \mathscr{M}^{\gm_m}_{X\times_k \gm_{m,k}}.
\end{align*}
For details, we refer the reader to \cite{guibert+loeser+merle-2006} and \cite[Proposition 5.2.1]{florian+julien-2021}. We generally call elements of Grothendieck rings \textit{virtual motives} in contrast to non-virtual motives in $\sh$. In the virtual setting, we also have a theory of motivic nearby cycles. More concretely, let $f \colon X \longrightarrow \mathbb{A}_k^1$ be a $k$-morphism and $U \subset X$ be a dense open subscheme of $X$ and $F = X \setminus U$ its closed complement endowed with the reduced structure. In \cite[Theorem 3.9]{guibert+loeser+merle-2006}, the authors construct a nearby cycles motive $\psi_{f,U}$, defined as certain limit of motivic zeta functions. The case $U = X$ was studied in various papers, see for instance, \cite{denef+loeser-1998}\cite{denef+loeser-1999}\cite{denef+loeser-2002}\cite{loeser-2000}\cite{loeser-2009}. We write $\psi_f$ in case $U = X$. There is a morphism 
\begin{equation*}
        \psi_f\colon \Mscr_X \longrightarrow \mathscr{M}_{X_{\sigma} \times_k \gm_m}^{\gm_{m,k}}
    \end{equation*}
    such that, for every proper morphism $g\colon Z \longrightarrow X$ with $Z$ smooth over $k$ and every dense open subset $U$ in $Z$, 
    \begin{equation*}
        \psi_f([U \longhookrightarrow Z \overset{g}{\longrightarrow} X]) = \breve{g}_{\sigma!}(\psi_{f \circ g,U}).
    \end{equation*}
where $\psi_{f \circ g,U}$ is the virtual nearby cycles, defined as certain limit of motivic zeta functions. 
\subsection{Some results on compatibility} In this section, we prove several results on compatibility of motivic nearby cycles morphisms and motivic nearby cycles functors. We also prove that motivic nearby cycles morphisms carry two "base change structures" with respect to the inverse image $g^*$ and the exceptional direct image $g_!$. These base change morphisms are isomorphisms if $g$ is respectively smooth and proper. That being said, nearby cycles morphisms share a lot of similarities with nearby cycles functors. In fact, proposition \ref{prop: motivic incarnation} shows that nearby cycles morphisms are non-virtual incarnations of nearby cycles functors. 

\begin{lem} \label{nearby cycles morphism commutes with proper push forward}
     Let $X,Y$ be $k$-varieties and $g\colon Y \longrightarrow X, f\colon X \longrightarrow \mathbb{A}_k^1$ be morphisms of $k$-varieties. If $g$ is proper, then we have a commutative diagram 
       \begin{equation*}
        \begin{tikzcd}[sep=large]
          \mathscr{M}_{Y} \arrow[r,"\psi_{f \circ g}"] \arrow[d,"g_!",swap] &   \mathscr{M}_{Y_{\sigma} \times_k \gm_{m,k}}^{\gm_{m}} \arrow[d,"\breve{g}_{\sigma!}"] \\ 
         \mathscr{M}_{X}   \arrow[r,"\psi_{f}"]&  \mathscr{M}_{X_{\sigma}\times_k \gm_{m,k}}^{\gm_{m}}.
        \end{tikzcd}
    \end{equation*}
\end{lem}

\begin{proof}
    Note that the Grothendieck ring of varieties $\mathscr{M}_X$ has the following generating class (see \cite[Theorem 5.1]{bittner-2004}): the classes of proper morphisms $h\colon Z \longrightarrow Y$ with $Z$ smooth over $k$ generate the ring $\Mscr_Y$. As a consequence, we may work from the beginning with such classes. We see that
    \begin{align*}
        \breve{g}_{\sigma!}\psi_{f \circ g}([h\colon Z \longrightarrow Y]) & = \breve{g}_{\sigma!}\breve{h}_{\sigma!}(\psi_{f \circ g \circ h,Z}) \\ 
        & = \breve{(g \circ h)}_{\sigma!}(\psi_{f \circ g \circ h,Z}) \\ & = \psi_f([g \circ h\colon Z \longrightarrow X]) \\ 
        & = \psi_f\big( g_!([h \colon Z \longrightarrow Y]) \big)
    \end{align*}
    as desired.
\end{proof}

The following lemma is known as the smooth base change for nearby cycles morphism.
\begin{lem} \label{smooth base change for nearby cycles morphism}
     Let $X,Y$ be $k$-varieties and $g\colon Y \longrightarrow X, f\colon X \longrightarrow \mathbb{A}_k^1$ be morphisms of $k$-varieties. If $g$ is smooth, then we have a commutative diagram 
       \begin{equation*}
        \begin{tikzcd}[sep=large]
          \mathscr{M}_{X} \arrow[r,"\psi_f"] \arrow[d,"g^*",swap] &   \mathscr{M}_{X_{\sigma} \times_k \gm_{m,k}}^{\gm_{m}} \arrow[d,"\breve{g}_{\sigma}^*"] \\ 
         \mathscr{M}_{Y}   \arrow[r,"\psi_{f \circ g}"]&  \mathscr{M}_{Y_{\sigma}\times_k \gm_{m,k}}^{\gm_{m}}.
        \end{tikzcd}
    \end{equation*}
\end{lem}
\begin{proof}
     The proof is identical with the proof of  \cite[List of Properties 8.4]{bittner-2005} though the definition of equivariant Grothendieck rings in \cite{bittner-2005} is slightly different from the ones that we are considering here.
\end{proof}
We are now going to prove that this nearby cycles morphism $\psi_f$ is compatible with the monodromic nearby cycles functor considered in \cite{florian+julien-2021}. The following is a motivic analogue of \cite[Proposition 3.17]{guibert+loeser+merle-2006}.

\begin{prop} \label{prop: motivic incarnation}
Let $X$ be a $k$-variety and $f\colon X \longrightarrow \Abb_k^1$ be a morphism of $k$-varieties, then the diagram 
    \begin{equation*}
        \begin{tikzcd}[sep=large]
            \Mscr_X \arrow[r,"\psi_f"] \arrow[d,"\chi_{X,c}",swap]& \Mscr_{X_{\sigma} \times_k \gm_{m,k}}^{\gm_{m}} \arrow[d,"\chi^{\gm_{m}}_{X_{\sigma},c}"] \\ 
            K_0(\shct(X)) \arrow[r,"\Psi^{\textnormal{mon}}_f"] & K_0(\qushct(X_{\sigma}))
        \end{tikzcd}
    \end{equation*}
    is commutative. 
\end{prop}

\begin{proof}
     Like in lemma \ref{nearby cycles morphism commutes with proper push forward}, we may work from the beginning with the classes of proper morphisms $g\colon Y \longrightarrow X$ with $Y$ smooth over $k$. Apply lemma \ref{proper base change theorem for monodromic nearby cycles functor} (note that $g$ is proper), we have an isomorphism 
    \begin{equation*}
        \begin{tikzcd}[sep=large]
            \sh(Y) \arrow[r,"\Psi^{\textnormal{mon}}_{f \circ g}"] \arrow[d,"g_!",swap] & \qush(Y_{\sigma}) \arrow[Rightarrow, shorten >=30pt, shorten <=30pt, dl,"\mu_g^{\textnormal{mon}}"] \arrow[d,"\breve{g}_{\sigma!}"] \\ 
         \sh(X) \arrow[r,"\Psi^{\textnormal{mon}}_{f}"] &  \qush(X_{\sigma}).
        \end{tikzcd}
    \end{equation*}
    Now it is a consequence of \cite[Theorem 1.2]{florian+julien-2013}\cite[Theorem 5.3.1]{florian+julien-2021} (see also \cite[Corollary 8.7]{ayoub+florian+julien-2017}) that $ \Psi_{f \circ g}^{\textnormal{mon}}(\mathds{1}_Y) = \chi_{Y_{\sigma},c}^{\gm_{m}}(\psi_{f \circ g})$. We see that
    \begin{align*}
        (\Psi^{\textnormal{mon}}_f \chi_{X,c})([g \colon Y \longrightarrow X]) & = \Psi^{\textnormal{mon}}_f  (g_!\mathds{1}_Y) =  \breve{g}_{\sigma!} \Psi^{\textnormal{mon}}_{f \circ g}(\mathds{1}_Y) \\
        & = \breve{g}_{\sigma!} \chi_{Y_{\sigma},c}^{\gm_{m}}(\psi_{f \circ g})  = \chi_{X_{\sigma},c}^{\gm_{m}}\breve{g}_{\sigma!}(\psi_{f \circ g}) \\ 
        & = \chi_{X_{\sigma},c}^{\gm_{m}}\psi_f([g \colon Y \longrightarrow X])
    \end{align*}
    as desired. 
\end{proof}

Combine the propositions above, we obtain the following, which is a motivic analogue of \cite[Théorème 2.9]{raibaut-2012}. In fact, if $k$ is embedded in $\mathbb{C}$, then under the Betti realization (see \cite{ayoub-2010}), the following result becomes \cite[Théorème 2.9]{raibaut-2012}.
\begin{cor} \label{cor: raibaut's theorem 2.9 analogue} 
  Let $X$ be a $k$-variety and $f\colon X \longrightarrow \mathbb{A}^1_k$ be a morphism of $k$-varieties, then the diagram
\begin{equation*}
    \begin{tikzcd}[sep=large]
        \mathscr{M}_{X } \arrow[r,"\psi_{f}"] \arrow[d,"\chi_{X,c}^{\gm_m}",swap] & \mathscr{M}_{X_{\sigma} \times_k \gm_{m,k}}^{\gm_m} \arrow[r,"\breve{f}_{\sigma!}"] \arrow[d,"\chi_{X_{\sigma},c}^{\gm_m}"] & \mathscr{M}_{\gm_{m,k}}^{\gm_m} \arrow[d,"\chi_{k,c}^{\gm_m}"] \\
        K_0(\qushct(X)) \arrow[r,"\Psi_{f}^{\textnormal{mon}}"] & K_0(\qushct(X_{\sigma})) \arrow[r,"\breve{f}_{\sigma!}"] & K_0(\qushct(\sigma)).
    \end{tikzcd}
\end{equation*}
is commutative.
\end{cor}
\begin{proof}
    This is a direct consequence of proposition \ref{prop: motivic incarnation} and functoriality, as noted in the preceding section.
\end{proof}

\subsection{Nearby cycles at infinity as a triangulated functor} Follow \cite{raibaut-2012}\cite{raibaut+fantini-2020}, we recall the following definition: let $U$ be a smooth $k$-variety and $f \colon U \longrightarrow \mathbb{A}^1_k$ be a morphism of $k$-varieties. Let $(X,u,\overline{f})$ be a compactification of $f$, i.e. a diagram
\begin{equation*}
    \begin{tikzcd}[sep=large]
        U \arrow[r,"u",hook] \arrow[dr,"f",swap] & X \arrow[d,"\overline{f}"] \\ 
        & \Abb_k^1 
    \end{tikzcd}
\end{equation*}
in which $u$ is an open dominant immersion and $\overline{f}$ is a proper morphism of $k$-varieties. We identify $f$ and $\overline{f}$ with the images of $t$ under the canonical morphisms $k[t] \longrightarrow \mathcal{O}_U(U)$ and $k[t] \longrightarrow \mathcal{O}_X(X)$ of $k$-algebras induced by $f$ and $\overline{f}$, respectively. For any rational point $a \in \mathbb{A}^1_k(k)$, we denote by $f-a\colon U \longrightarrow \mathbb{A}_k^1$ and $\overline{f}-a\colon U \longrightarrow \mathbb{A}_k^1$ the induced morphisms by $f$ and $\overline{f}$, respectively. The difference 
\begin{equation*}
     \psi^{\infty}_{f,a} \coloneqq \breve{(\overline{f}-a)}_{\sigma!}(\psi_{\overline{f}-a,U}) - \breve{(f-a)}_{\sigma!}(\psi_{f-a}) \in \Mscr_{a \times_k \gm_{m,k}}^{\gm_m}.
\end{equation*}
is called the \textit{nearby cycles at infinity for the value $a$}, which does not depend on the choice of the compactification $(X,\hat{f},u)$ thanks to \cite[Théorème 4.2 and Corollaire 4.4]{raibaut-2012}. By lemma \ref{nearby cycles morphism commutes with proper push forward}, we reprove the independency of $\psi^{\infty}_{f,a}$ avoiding using \cite[Théorème 2.8]{raibaut-2012}. 
\begin{lem}
    Keep the notation, we have an equality 
    \begin{equation*}
        \psi^{\infty}_{f,a} = \psi_{\id}((f-a)_![U]) - \breve{(f-a)}_{\sigma!}\psi_{f-a}([U]).
    \end{equation*}
    in $\Mscr^{\gm_m}_{a \times_k \gm_{m,k}}$, where $[U]$ is the class of $\id \colon U \longrightarrow U$ in the ring $\Mscr_U$. 
\end{lem}
The following serves as a version at infinity of \cite[Theorem 1.2]{florian+julien-2013}\cite[Theorem 5.3.1]{florian+julien-2021}.
\begin{prop} \label{prop: compatibility of nearby functors with nearby cycles}
Let $U$ be a smooth $k$-variety and $f \colon U \longrightarrow \Abb_k^1$ be a morphism of $k$-varieties. Let $a \in \mathbb{A}_k^1(k)$ be a rational point. There is an equality
\begin{equation*}
[\Psi^{\textnormal{mon},\infty}_{\id,f-a}(\mathds{1}_U)] = \chi^{\gm_{m}}_{k,c}(\psi^{\infty}_{f,a})
\end{equation*}
in $K_0(\qushct(k))$. In particular, taking the inverse image of the unit section, there is an equality
\begin{equation*}
[\Psi^{\infty}_{\id,f-a}(\mathds{1}_U)] = \chi_{k,c}(\psi^{\infty}_{f,a})
\end{equation*}
in $K_0(\shct(k))$. 
\end{prop}

\begin{proof}
    To simplify our notation and without loss of generality, we may assume that $a=0$. By the existence of the monodromic nearby cycles functor $\Psi^{\textnormal{mon},\infty}_{\id,f}$, we have the following distinguished triangle
    \begin{equation*}
        \breve{f}_{\sigma!}\Psi_{f}^{\textnormal{mon}}\longrightarrow  \Psi_{\id}^{\textnormal{mon}}f_! \longrightarrow \Psi^{\textnormal{mon},\infty}_{\id,f} \longrightarrow +1.
    \end{equation*}
    At the level of the Grothendieck rings, it results in 
    \begin{equation*}
         [\Psi^{\textnormal{mon},\infty}_{\id,f}(\mathds{1}_U) ]=  [\Psi_{\id}^{\textnormal{mon}}f_!(\mathds{1}_U)] - [\breve{f}_{\sigma!}\Psi_{f}^{\textnormal{mon}}(\mathds{1}_U)].
    \end{equation*}
    We calculate each term by functoriality and proposition \ref{prop: motivic incarnation}. 
    \begin{equation*}
        [\Psi_{\id}^{\textnormal{mon}} f_!(\mathds{1}_U)] = [\Psi_{\id}^{\textnormal{mon}}\chi_{\mathbb{A}^1_k,c}([f \colon U \longrightarrow \mathbb{A}_k^1]) ]= \chi_{k,c}^{\gm_m}\psi_{\id}(f_![U]).
    \end{equation*}
    Similarly, 
    \begin{equation*}
        [\breve{f}_{\sigma!}\Psi_{f}^{\textnormal{mon}}(\mathds{1}_U)] =  [\breve{f}_{\sigma!}\chi^{\gm_m}_{U_{\sigma},c}(\psi_f[U])] =  \chi^{\gm_m}_{k,c}[\breve{f}_{\sigma!}(\psi_f[U])].
    \end{equation*}
    The desired equality is obtained by additivity. 
\end{proof}
The following corollary is an analogue of \cite[Theorem 3.3]{raibaut+fantini-2020} in the world of motives. 
\begin{cor} \label{cor: analogue of raibaut+fantini-2020, theorem 3.3}
    Let $f\colon U = \mathbb{A}^n_k \longrightarrow \mathbb{A}_k^1$ be a morphism and $a \in \mathbb{A}^1_k(k)$ be a rational point. Suppose that $a$ does not lie in the discriminant of $f$. Then there is a distinguished triangle 
    \begin{equation*}
        \breve{(f-a)}_{\sigma!}(\mathds{1}_{U_a}) \longrightarrow \Psi^{\textnormal{mon}}_{\id}(f-a)_!(\mathds{1}_U) \longrightarrow \Psi^{\textnormal{mon},\infty}_{\id,f-a}(\mathds{1}_U) \longrightarrow +1
    \end{equation*}
    in $\qushct(k)$. In particular, 
    \begin{equation*}
       [\Psi^{\textnormal{mon},\infty}_{\id,f-a}(\mathds{1}_U)] = [\Psi^{\textnormal{mon}}_{\id}(f-a)_!(\mathds{1}_U)] - [\breve{(f-a)}_{\sigma!}(\mathds{1}_{U_a})]  
    \end{equation*}
    in $K_0(\qushct(k))$.
\end{cor}
\begin{proof}
    We may suppose that $a=0$. It is sufficient to prove that $\Psi_f(\mathds{1}_{U_{\eta}}) \simeq \mathds{1}_{U_{\sigma}}$. This is precisely \cite[Proposition 3.4]{ayoub+florian+julien-2017} (see also \cite[Théorème 10.6]{ayoub-2014}).
\end{proof}

\section{Motivic bifurcation sets} 
In this section, we assume that $k$ is an algebraically closed field of characteristic zero. We consider three bifurcation sets 
\begin{align*}
    B^{\textnormal{Rai}}_f &  = \left \{a \in \mathbb{A}^1_k(k) \mid \psi^{\infty}_{f,a} \neq 0 \in \mathscr{M}_k \right \} \cup \mathrm{disc}(f) \\ 
       B^{\textnormal{coh}}_f &  = \left \{a \in \mathbb{A}^1_k(k) \mid \Psi^{\infty}_{\id,f-a}(\mathds{1}) \not\simeq 0  \ \textnormal{in} \ \sh(k)\right \} \cup \mathrm{disc}(f) \\ 
          B^{\textnormal{Eu}}_f &  = \left \{a \in \mathbb{A}^1_k(k) \mid [\Psi^{\infty}_{\id,f-a}(\mathds{1})] \neq 0 \ \textnormal{in} \ K_0(\sh(k)) \right \} \cup \mathrm{disc}(f) 
\end{align*}
where $\mathrm{disc}(f)$ denotes the discriminant of $f$. We emphasize that $B^{\textnormal{Rai}}_f, B^{\textnormal{Eu}}_f$ measure the Euler characteristics of nearby cycles at infinity while $B^{\textnormal{coh}}_f$ measures cohomology of nearby cycles at infinity. The Raibaut bifurcation set $B^{\textnormal{Rai}}_f$ is finite according to \cite[Théorème 4.13]{raibaut-2012} and by the realization in proposition \ref{prop: compatibility of nearby functors with nearby cycles} we know that $B^{\textnormal{Eu}}_f \subset B^{\textnormal{Rai}}_f \cap B^{\textnormal{coh}}_f$, which implies that $B^{\textnormal{Eu}}_f$ is finite as well. The following question is natural \medskip \\ 
\textbf{Question}. Is $B^{\textnormal{coh}}_f$ finite for any $f$? In other words, does nearby cycles at infinity vanish comologically outsite a finite set? \medskip \\ 
In this section, we show that the question above has positive answer for Laurent polynomials non-degenerate with respect to their Newton polyhedra. Moreover, we show that $B^{\textnormal{Rai}}_f = B^{\textnormal{coh}}_f = B^{\textnormal{Eu}}_f = \left \{0 \right \}$ with $f$ a monomial of the form $x_1\cdots x_n$. Since $B^{\textnormal{Rai}}_f = \left \{0 \right \}$ by \cite[Proposition 4.18]{raibaut-2012}, it suffices to show that $B^{\textnormal{coh}}_f = \left \{0 \right \}$. We expect the same result holds for general monomials and possibly homogeneous polynomials. First we give some simple criterion to detect non-bifurcate points. The following definition is motivated from the topological point of view.

\begin{defn} \label{defn: fibration}
Let $f \colon X \longrightarrow \mathbb{A}^1_k$ be a morphism of $k$-varieties. Let $a \in \mathbb{A}^1_k(k) = k$ be a rational point. We say that $a$ is a \textit{typical value} if there exists an open subset $u \colon U \longhookrightarrow \mathbb{A}^1_k$ containing $a$ and there exists a commutative diagram 
\begin{equation*}
    \begin{tikzcd}[sep=large]
      U \times_k F \arrow[r,"\varphi",] \arrow[rd,"p",swap] & X\times_{\mathbb{A}^1_k} U \arrow[r,"u_X"] \arrow[d,"f_U"] & X \arrow[d,"f"] \\ 
      & U \arrow[r,hook,"u"] & \mathbb{A}_k^1.
    \end{tikzcd}
\end{equation*}
with $p$ the projection, $F$ a $k$-variety and $\varphi$ an isomorphism. Otherwise, we say that $a$ is an \textit{atypical value} of $f$.
\end{defn}
We are going to prove that if a morphism $f \colon X \longrightarrow \mathbb{A}^1_k$ admits a typical value $a \in \mathbb{A}^1_k(k)$, then the nearby cycles functor at infinity vanishes. This fits our intuition that there should be no nearby cycles at infinity when, in a neighborhood of infinity, the regular becomes a fibration. However, we cannot prove the result in full generality. The reason is in the topological world (when $k = \mathbb{C}$), a function $f \colon X \longrightarrow \mathbb{C}$ becomes a trivial $C^{\infty}$-fibration; that being said, the fiber may not be induced by an algebraic variety, hence we have to restrict a smaller class of fibers $F$. Here we propose a class named weakly elementary neighborhood (motived by elementary fibration considered in \cite{sga4}). This includes basic schemes such as (finite, smooth extensions of) $\mathbb{A}^n_k \setminus \left\{ \textnormal{finite points} \right \}$, $\mathbb{G}_{m,k}^n$, etc.
\begin{defn} \label{defn: weakly elementary fibration}
A $k$-variety $F$ is a \textit{weakly elementary fibration} if it is a part of a commutative diagram
    \begin{equation*}
        \begin{tikzcd}[sep=large]
       F \arrow[r,"j"] \arrow[rd,"p",swap]&  \hat{F} \arrow[d,"\hat{p}"] &  Z \arrow[ld,"e"] \arrow[l,"i",swap]  \\ 
          & k &
        \end{tikzcd}
    \end{equation*}
    with $j$ an open immersion, $i$ the closed complement (with reduced schematic structure), $\hat{p}$ proper smooth, $e$ finite smooth. An \textit{weakly elementary neighborhood} is defined to be a weakly elementary fibration follows by a smooth proper morphism.
\end{defn}
\begin{rmk}
    In the definition above, if we require $\hat{p}$ to have relative dimension $1$, $e$ to be \'etale, then $F$ is called a elementary fibration in \cite[D\'efinition 3.1, Expos\'e XI]{sga4}. The fundamental result in \cite[Expos\'e XI]{sga4} states that any scheme can be covered by Artin neighborhoods, namely, successive composition of elementary fibrations. 
\end{rmk}
\begin{theorem} \label{thm: detecting typical values}
   Let $f \colon X \longrightarrow \mathbb{A}_k^1$ be a morphism of $k$-varieties and $a \in \mathbb{A}^1_k(k)$ be a value. Assume that $a$ is a typical value with fiber $F$. Suppose that $F$ is a weakly elementary neighborhood over $k$ then $\Psi^{\infty}_{\id,f-a}(f-a)_{\eta}^* = 0$.
\end{theorem}

\begin{proof}
    It is enough to assume that $a = 0 \in U$. We have to show that if $0$ admits a locally trivial neighborhood, then 
    \begin{equation*}
        u_{\sigma}^*f_{\sigma!}\Psi_ff_{\eta}^*\longrightarrow u_{\sigma}^*\Psi_{\id}f_{\eta!}f_{\eta}^*
    \end{equation*}
    is an isomorphism since $u_{\sigma}=\id$. There is a commutative diagram 
    \begin{equation*}
        \begin{tikzcd}[sep=large]
           u_{\sigma}^*f_{\sigma!}\Psi_ff_{\eta}^* \arrow[r] \arrow[d] &  u_{\sigma}^*\Psi_{\id}f_{\eta!}f_{\eta}^* \arrow[d] \\ 
           f_{U\sigma!}u_{X\sigma}^*\Psi_f f_{\eta}^* \arrow[d] & \Psi_{u}u_{\eta}^*f_{\eta!}f_{\eta}^*  \arrow[d] \\ 
           f_{U\sigma!}\Psi_{f \circ u_X}u_{X\eta}^*f_{\eta}^* \arrow[r] & \Psi_uf_{U\eta!}u^*_{X\eta}f_{\eta}^*.
        \end{tikzcd}
    \end{equation*}
    Thus, it suffices to show that
    \begin{equation*}
         f_{U\sigma!}\Psi_{f \circ u_X}f_{U\eta}^*u_{\eta}^* \longrightarrow \Psi_uf_{U\eta!}f_{U\eta}^*u_{\eta}^*
    \end{equation*}
    is an isomorphism. Since $\varphi$ is an isomorphism, we can show that
        \begin{equation*}
         f_{U\sigma!}\Psi_{u \circ f_U}\varphi_{\eta*}\varphi_{\eta}^*f_{U\eta}^*u_{\eta}^* \longrightarrow \Psi_uf_{U\eta!}\varphi_{\eta*}\varphi_{\eta}^*f_{U\eta}^*u_{\eta}^*,
    \end{equation*}
    which boils down to proving that 
     \begin{equation*}
        p_{\sigma!}p_{\sigma}^*\Psi_{u }u_{\eta}^* \simeq  p_{\sigma!}\Psi_{u \circ p}p_{\eta}^*u_{\eta}^*  \longrightarrow  \Psi_u p_{\eta!}p_{\eta}^*u_{\eta}^* 
    \end{equation*}
    is an isomorphism. In case $F$ is a weakly elementary fibration, we consider a 
    \begin{equation*}
        \begin{tikzcd}[sep=large]
       U \times_k F \arrow[r,"j"] \arrow[rd,"p",swap]& U \times_k \hat{F} \arrow[d,"\hat{p}"] & U \times_k Z \arrow[ld,"e"] \arrow[l,"i",swap]  \\ 
          & U &
        \end{tikzcd}
    \end{equation*}
    with $\hat{p}$ smooth, projective and $e$ finite, smooth. The localization sequence implies the existence of a triangle
    \begin{equation*}
        p_!p^* \longrightarrow \hat{p}_!\hat{p}^* \longrightarrow e_!e^* \longrightarrow +1.
    \end{equation*}
    Consequently, we obtain a commutative diagram
    \begin{equation*}
        \begin{tikzcd}[sep=large]
          p_{\sigma!}p_{\sigma}^*\Psi_{u }u_{\eta}^*\arrow[r] \arrow[d] & \hat{p}_{\sigma!}\hat{p}_{\sigma}^*\Psi_{u }u_{\eta}^* \arrow[r] \arrow[d] & e_{\sigma!}e_{\sigma}^*\Psi_{u }u_{\eta}^* \arrow[d] \arrow[r] & +1 \\ 
             \Psi_u p_{\eta!}p_{\eta!}^* u_{\eta}^* \arrow[r] & \Psi_u\hat{p}_{\eta!}\hat{p}_{\eta!}^*u_{\eta}^* \arrow[r]  & \Psi_u e_{\eta!}e_{\eta}^*u_{\eta}^* \arrow[r] & +1.  
        \end{tikzcd}
    \end{equation*}
    Since both $\hat{p},e$ are proper and smooth, we see that the two vertical arrows corresponding to $\hat{F},Z$ are isomorphisms (by smooth and proper base change for nearby cycles functors), which forces the first vertical arrow to be an isomorphism.
\end{proof}

\begin{cor}
   Let $f(x_1,...,x_n) = x_1\cdots x_i$ (with $1\leq i \leq n$) viewed as a morphism $f \colon X = \mathbb{A}^n_k \longrightarrow \mathbb{A}^1_k$, then for all $a \in \mathbb{A}^1_k(k)$, $\Psi^{\infty}_{\id,f-a}(f-a)^*_{\eta} \simeq 0$. 
\end{cor}

\begin{proof}
    Over $\gm_{m,k} \longhookrightarrow \mathbb{A}_k^1$, the morphism $f$ is a trivial fibration with fiber $\mathbb{A}_k^{n-i} \times_k \gm_{m,k}^{i-1}$ so it is an immediate consequence of the previous theorem. At $a = 0$, the result is a consequence of the motivic Davison-Meinhardt conjecture proved in \cite{florian+julien-2023}. 
\end{proof}

We end this section by showing that $a = 0$ is an atypical value for the Broughton's example $f(x,y) = x(xy-1)$. For a $k$-variety $p \colon X \longrightarrow \Spec(k)$, we denote by $\textnormal{M}^{\vee}_c(X) = p_!(\mathds{1}_X)$ its cohomological motive with compact support. 
\begin{prop}
 Let $f(x,y) = x(xy-1)$ viewed as a morphism $f \colon X = \mathbb{A}_k^2 \longrightarrow \mathbb{A}_k^1$, then $\Psi_{\id,f}^{\infty}(\mathds{1}_{X_{\eta}}) \not\simeq 0$ and $\Psi_{\id,f-a}^{\infty}(f-a)^*_{\eta} \simeq 0$ for $a \neq 0$. In particular, $0$ is the only atypical value of $f$. 
\end{prop}

\begin{proof}
    For values $a \neq 0$, the proof is similar to the proof of theorem \ref{thm: detecting typical values}. Let $a = 0$, since $f$ is a smooth morphism, we see that 
    \begin{equation*}
        \Psi_f(\mathds{1}_{X_{\eta}}) \simeq \Psi_ff_{\eta}^*(\mathds{1}_{\gm_{m,k}}) \simeq f_{\sigma}^*\Psi_{\id}(\mathds{1}_{\gm_{m,k}}) \simeq f_{\sigma}^*(\mathds{1}_{\sigma}) \simeq \mathds{1}_{X_{\sigma}}
    \end{equation*}
    since $\Psi_{\id}$ is unital by \cite[Corollaire 3.5.19]{ayoub-thesis-2}. The special fiber $X_{\sigma}$ is isomorphism to a disjoint union $\mathbb{A}_k^1\sqcup \gm_{m,k}$. Therefore, we get 
    \begin{equation*}
        f_{\sigma!}\Psi_f(\mathds{1}_{X_{\eta}}) \simeq f_{\sigma!}(\mathds{1}_{X_{\sigma}}) \simeq \textnormal{M}^{\vee}_c(\mathbb{A}^1_k) \oplus \textnormal{M}^{\vee}_c(\gm_{m,k}).
    \end{equation*}
    On the other hand, the generic fiber is isomorphic to $X_1 \times_k \gm_{m,k} \simeq \gm_{m,k} \times \gm_{m,k}$ (with $X_1$ the fiber over $1$) via the parametrization (at the level of points)
   \begin{align*}
       X_{\eta} = \left \{(x,y) \mid x(xy-1) = t \neq 0 \right \} & \longrightarrow X_1 \times_k \gm_{m,k} \\
       (x,y) & \longrightarrow ((t^{-1}x,ty),t).
   \end{align*}
   We argue as in \cite[Proposition 5.2]{florian+julien-2023} to deduce that $\Psi_{\id}f_{\eta!}f_{\eta}^* \simeq f_{1!}f_1^*$. Thus, we obtain that 
   \begin{equation*}
       \Psi_{\id}f_{\eta!}(\mathds{1}_{X_{\eta}}) \simeq f_{1!}(\mathds{1}_{X_1}) \simeq \textnormal{M}^{\vee}_c(\gm_{m,k} \times_k \gm_{m,k}) \simeq \textnormal{M}^{\vee}_c(\gm_{m,k}) \otimes \textnormal{M}^{\vee}_c(\gm_{m,k})
   \end{equation*}
   by the Kunneth formula. One sees immediately that $f_{\sigma!}\Psi_f(\mathds{1}_{X_{\eta}})$ is not isomorphic to $\Psi_{\id}f_{\eta!}(\mathds{1}_{X_{\eta}})$. This is even true at the level of $K_0(\shct(\sigma))$. Indeed, we know that $\textnormal{M}^{\vee}_c(\mathbb{A}_k^1) = \mathds{1}_k(-1)[-2]$ and the localization
   \begin{equation*}
       \textnormal{M}^{\vee}_c(\gm_{m,k}) \longrightarrow  \mathds{1}_k(-1)[-2] \longrightarrow \mathds{1}_k \longrightarrow +1.
   \end{equation*}
   implies that $[\textnormal{M}^{\vee}_c(\gm_{m,k})] = [\mathds{1}_k(-1)] - [\mathds{1}_k]$ in $K_0(\shct(\sigma))$. Consequently, we can calculate 
   \begin{equation} \label{eq:3.1}
   \begin{split} 
      [ f_{\sigma!}\Psi_f(\mathds{1}_{X_{\eta}}) ] & = 2[\mathds{1}_k(-1)] - [\mathds{1}_k] \\ 
      [\Psi_{\id}f_{\eta!}(\mathds{1}_{X_{\eta}}) ] & = ([\mathds{1}_k(-1)]-[\mathds{1}_k])^2 = [\mathds{1}_k(-2)] - 2[\mathds{1}_k(-1)] + [\mathds{1}_k].
      \end{split} 
   \end{equation}
  Under Betti realization functor (see \cite{ayoub-2010}) and Euler characteristics, the term $[ f_{\sigma!}\Psi_f(\mathds{1}_{X_{\eta}}) ]$ equals $1$ but the term $ [\Psi_{\id}f_{\eta!}(\mathds{1}_{X_{\eta}}) ]$ equals $0$ (this reflects the fact that we have Betti numbers $b_0(f^{-1}(0))=2,b_1(f^{-1}(0))=1$ but $b_0(f^{-1}(a)) = b_1(f^{-1}(a))=1$ if $a \neq 0$).
\end{proof}

\section{Appendix: Elimination of quasi-projectiveness}
In this appendix, we eliminate the quasi-projectiveness assumption of varieties in the theory of motivic nearby cycles functors, which is implicitly used in theorem \ref{existence of nearby cycles functors at infty}. In \cite{ayoub-thesis-2}, the author use the smooth-closed factorization of morphisms, which is available for morphisms between quasi-projectiveness varieties. Here we use the proper-open factorization, which is available for morphisms between general varieties, thanks to the Nagata compactification theorem. We begin with a simple lemma.
\begin{lem} \label{proper base change for nearby cycles functor}
The base change morphism $\beta_g\colon \Psi_f g_{\eta*} \longrightarrow g_{\sigma*}\Psi_{f \circ g}$ is an isomorphism if $g$ is proper.     
\end{lem}
\begin{proof}
The base change morphism $\beta_g \colon \Psi_f g_{\eta*} \longrightarrow g_{\sigma *}\Psi_{f \circ g}$ is given by the diagram
\begin{equation*}
    \begin{tikzcd}[column sep = large, row sep = 1.2cm]
       \sh(Y_{\eta}) \arrow[d,"g_{\eta*}",swap] \arrow[r,"(\theta_{f \circ g})^*p_I^*"] & \shbb(\mathscr{R}_Y,I) \arrow[r,"(j_{f \circ g} \circ \theta_{f \circ g})_*"] \arrow[d,"g_{\eta *}"] & \shbb(Y,I) \arrow[d,"g_*"] \arrow[r,"(i_{f \circ g})^*"] &  \shbb(Y_{\sigma},I) \arrow[d,"g_{\sigma *}"] \arrow[r,"(p_I)_{\#}"] & \sh(Y_{\sigma}) \arrow[d,"g_{\sigma *}"] \\ 
       \sh(X_{\eta}) \arrow[Rightarrow, shorten >=30pt, shorten <=30pt, ur,"\Ex^*_*"]   \arrow[r,"(\theta_f)^*p_I^*"] & \shbb(\mathscr{R}_X,I) \arrow[r,"(j_f \circ \theta_f)_*"] \arrow[Rightarrow, shorten >=30pt, shorten <=30pt, ur,"\Ex_{**}"]  & \shbb(X,I) \arrow[r,"(i_f)^*"] \arrow[Rightarrow, shorten >=30pt, shorten <=30pt, ur,"\Ex^*_*"]  &  \shbb(X_{\sigma},I) \arrow[r,"(p_I)_{\#}"] \arrow[Rightarrow, shorten >=30pt, shorten <=30pt, ur,"\Ex_{\#*}"]  & \sh(X_{\sigma}).
    \end{tikzcd}
\end{equation*} 
  In the defining diagram of $\beta_g$, both $\Ex_*^*$ are isomorphisms due to proposition \ref{proper base change for derivators}. The morphism $\Ex_{\#*}$ is an isomorphism thanks to proposition \ref{direct images commute with homotopy colimits}.
\end{proof}

\begin{prop} \label{magic diagram}
Let $X$ be a $k$-variety and $f\colon X \longrightarrow \mathbb{A}_k^1$ be a morphism of $k$-varieties. Let $\overline{X},Y,\overline{Y}$ be $k$-varieties. Given a cartesian diagram of $k$-varieties
\begin{equation*}
    \begin{tikzcd}[sep=large]
        \overline{Y} \arrow[d,"k",swap] \arrow[rd,phantom,"\textnormal{(C)}"] \arrow[r,"\overline{u}"] & \overline{X} \arrow[d,"h"] \\ 
        Y \arrow[r,"u"] & X 
    \end{tikzcd}
\end{equation*}
with $u,\overline{u}$ are smooth morphisms, then the cube 
    \begin{equation} \label{magic cube}
       \begin{tikzcd}[sep=scriptsize]
         \sh(Y_{\eta}) \arrow[rr,"\Psi_{f \circ u}"] \arrow[rd,"u_{\eta \#}"] & & \sh(Y_{\sigma}) \arrow[rd,"u_{\sigma \#}"] & \\ 
           & \sh(X_{\eta}) & & \sh(X_{\sigma})   \\ 
         \sh(\overline{Y}_{\eta}) \arrow[uu,"k_{\eta *}"]  \arrow[rr,"\Psi_{f \circ u \circ k }",dashed,pos=0.2] \arrow[rd,"\overline{u}_{\eta \#}",swap]  & & \sh(\overline{Y}_{\sigma}) \arrow[rd,"\overline{u}_{\sigma \#}",dashed]  \arrow[uu,"k_{\sigma *}",pos=0.3,dashed]   & \\ 
            & \sh(\overline{X}_{\eta}) \arrow[uu,"h_{\eta *}",pos=0.3,crossing over]  \arrow[rr,"\Psi_{f \circ h}"] & & \sh(\overline{X}_{\sigma}) \arrow[uu,"h_{\sigma *}"]  \arrow[from=2-2,to=2-4,"\Psi_{f}",crossing over, pos=0.2] 
            \end{tikzcd}
   \end{equation}
   is commutative. In other words, the diagram 
  \begin{equation*}
        \begin{tikzcd}[sep=large]
            h_{\sigma *}\overline{u}_{\sigma \#}\Psi_{f \circ u \circ k} \arrow[r]  &  h_{\sigma *}\Psi_{f \circ h}\overline{u}_{\eta \#} \\ 
             u_{\sigma \#}k_{\sigma *}\Psi_{f \circ u \circ k} \arrow[u,"(\Ex_{\#*})"]   & \Psi_{f} h_{\eta *}\overline{u}_{\eta \#}  \arrow[u] \\ 
             u_{\sigma \#}\Psi_{f \circ u}k_{\eta *}\arrow[r] \arrow[u] &   \Psi_{f}u_{\eta \#}k_{\eta *}. \arrow[u,"(\Ex_{\#*})",swap]
        \end{tikzcd}
    \end{equation*}
    is commutative. If moreover, $u,\overline{u}$ are open immersions and $h,k$ are proper, then we can remove "cartesian" in the statement above.
\end{prop}

\begin{proof}
The second case can be obtained from the first one by the interchange law so we can focus on the case when $\textnormal{(C)}$ is cartesian. The cube \ref{magic cube} is commutative if the cube
      \begin{equation*}
       \begin{tikzcd}[sep=scriptsize]
         \sh(Y_{\eta}) \arrow[rr,"\Psi_{f \circ u}"] \arrow[rd,"u_{\eta \#}"] \arrow[dd,"k_{\eta}^*",swap] & & \sh(Y_{\sigma}) \arrow[rd,"u_{\sigma \#}"] \arrow[dd,"k_{\sigma}^*",pos=0.7,dashed] & \\ 
           & \sh(X_{\eta})  & & \sh(X_{\sigma}) \arrow[dd,"h_{\sigma}^*"]  \\ 
         \sh(\overline{Y}_{\eta}) \arrow[rr,"\Psi_{f \circ u \circ k}",dashed,pos=0.2] \arrow[rd,"\overline{u}_{\eta \#}",swap]  & & \sh(\overline{Y}_{\sigma}) \arrow[rd,"\overline{u}_{\sigma \#}",dashed]  & \\ 
            & \sh(\overline{X}_{\eta}) \arrow[rr,"\Psi_{f \circ h}"] & & \sh(\overline{X}_{\sigma}) \arrow[from=2-2,to=2-4,"\Psi_{f}",crossing over, pos=0.2] \arrow[from=2-2,to=4-2,"h_{\eta}^*",pos=0.3,crossing over],
       \end{tikzcd}
   \end{equation*}
   is commutative according to \cite[Corollaire 1.1.14]{ayoub-thesis-1}. Note that by reversing the direction of $2$-morphisms, we obtain a dual version of \cite[Corollaire 1.1.14]{ayoub-thesis-1} and using this, we reduce to the commutativity of 
      \begin{equation*}
       \begin{tikzcd}[sep=scriptsize]
         \sh(Y_{\eta}) \arrow[rr,"\Psi_{f \circ u}"]  \arrow[dd,"k_{\eta}^*",swap] & & \sh(Y_{\sigma}) \arrow[dd,"k_{\sigma}^*",pos=0.7,dashed] & \\ 
           & \sh(X_{\eta})  \arrow[lu,"u_{\eta}^*",swap]  & & \sh(X_{\sigma}) \arrow[dd,"h_{\sigma}^*"] \arrow[lu,"u_{\eta}^*",swap] \\ 
         \sh(\overline{Y}_{\eta}) \arrow[rr,"\Psi_{f \circ u \circ k}",dashed,pos=0.2]  & & \sh(\overline{Y}_{\sigma})  & \\ 
            & \sh(\overline{X}_{\eta}) \arrow[lu,"\overline{u}_{\eta}^*"]  \arrow[rr,"\Psi_{f \circ h}"] & & \sh(\overline{X}_{\sigma}) \arrow[lu,"\overline{u}_{\sigma}^*"] \arrow[from=2-2,to=2-4,"\Psi_{f}",crossing over, pos=0.2] \arrow[from=2-2,to=4-2,"h_{\eta}^*",pos=0.3,crossing over] ,
       \end{tikzcd}
   \end{equation*}
   which is trivial. 
\end{proof}

\begin{cor} \label{second magic diagram}
Let $X$ be a $k$-variety and $f\colon X \longrightarrow \mathbb{A}_k^1$ be a morphism of $k$-varieties. Let $\overline{X},Y,\overline{Y}$ be $k$-varieties. Given a cartesian diagram of $k$-varieties
\begin{equation*}
    \begin{tikzcd}[sep=large]
        \overline{Y} \arrow[d,"k",swap] \arrow[r,"\overline{u}"] & \overline{X} \arrow[d,"h"] \\ 
        Y \arrow[r,"u"] & X 
    \end{tikzcd}
\end{equation*}
with $u,\overline{u}$ are smooth morphisms, then the cube 
    \begin{equation} \label{second magic cube}
       \begin{tikzcd}[sep=scriptsize]
         \sh(Y_{\eta}) \arrow[rr,"\Psi_{f \circ u}"] \arrow[rd,"u_{\eta \#}"] & & \sh(Y_{\sigma}) \arrow[rd,"u_{\sigma \#}"] & \\ 
           & \sh(X_{\eta}) & & \sh(X_{\sigma})   \\ 
         \sh(\overline{Y}_{\eta}) \arrow[uu,"k_{\eta !}"]  \arrow[rr,"\Psi_{f \circ u \circ k }",dashed,pos=0.2] \arrow[rd,"\overline{u}_{\eta \#}",swap]  & & \sh(\overline{Y}_{\sigma}) \arrow[rd,"\overline{u}_{\sigma \#}",dashed]  \arrow[uu,"k_{\sigma !}",pos=0.3,dashed]   & \\ 
            & \sh(\overline{X}_{\eta}) \arrow[uu,"h_{\eta !}",pos=0.3,crossing over]  \arrow[rr,"\Psi_{f \circ h}"] & & \sh(\overline{X}_{\sigma}) \arrow[uu,"h_{\sigma !}"] \arrow[from=2-2,to=2-4,"\Psi_{f}",crossing over, pos=0.2] 
            \end{tikzcd}
   \end{equation}
   is commutative. In other words, the diagram 
  \begin{equation*}
        \begin{tikzcd}[sep=large]
            h_{\sigma !}\overline{u}_{\sigma \#}\Psi_{f \circ u \circ k} \arrow[r]  &  h_{\sigma!}\Psi_{f \circ h}\overline{u}_{\eta \#} \\ 
             u_{\sigma \#}k_{\sigma !}\Psi_{f \circ u \circ k} \arrow[u,"(\Ex_{\#!})"]   & \Psi_{f} h_{\eta !}\overline{u}_{\eta \#}  \arrow[u] \\ 
             u_{\sigma \#}\Psi_{f \circ u}k_{\eta !}\arrow[r] \arrow[u] &   \Psi_{f}u_{\eta \#}k_{\eta !}. \arrow[u,"(\Ex_{\#!})",swap]
        \end{tikzcd}
    \end{equation*}
    is commutative. 
\end{cor}

\begin{proof}
Let 
\begin{equation*}
\begin{tikzcd}[sep=large]
    \overline{X} \arrow[r,"u_X"] \arrow[rd,"h",swap] & \hat{X} \arrow[d,"\hat{h}"] \\ 
    & X 
    \end{tikzcd}
\end{equation*}
be a compactification of $h$. By base change, we obtain cartesian squares where the left column is a compactification of $k$
\begin{equation*}
    \begin{tikzcd}[sep=large]
        \overline{Y} \arrow[dd,"k",swap,bend right = 50] \arrow[d,"u_Y",swap] \arrow[r,"\overline{u}"] & \overline{X} \arrow[d,"u_X"] \arrow[dd,"h",bend left = 50] \\ 
          \hat{Y} \arrow[d,"\hat{k}",swap] \arrow[r,"\hat{u}"] & \hat{X} \arrow[d,"\hat{h}"] \\ 
        Y \arrow[r,"u"] & X 
    \end{tikzcd}
\end{equation*}
The cube \ref{second magic cube} is decomposed into two cubes 
 \begin{equation*} 
       \begin{tikzcd}[sep=scriptsize]
         \sh(Y_{\eta}) \arrow[rr,"\Psi_{f \circ u}"] \arrow[rd,"u_{\eta \#}"] & & \sh(Y_{\sigma}) \arrow[rd,"u_{\sigma \#}"] & \\ 
           & \sh(X_{\eta}) & & \sh(X_{\sigma})   \\ 
         \sh(\hat{Y}_{\eta}) \arrow[uu,"\hat{k}_{\eta *}"]  \arrow[rr,"\Psi_{f \circ u \circ \hat{k}}",dashed,pos=0.2] \arrow[rd,"\hat{u}_{\eta \#}",swap]  & & \sh(\hat{Y}_{\sigma}) \arrow[rd,"\hat{u}_{\sigma \#}",dashed]  \arrow[uu,"\hat{k}_{\sigma *}",pos=0.3,dashed]   & \\ 
            & \sh(\hat{X}_{\eta}) \arrow[uu,"\hat{h}_{\eta *}",pos=0.3,crossing over]   & & \sh(\hat{X}_{\sigma}) \arrow[uu,"\hat{h}_{\sigma *}"] \\ 
             \sh(\overline{Y}_{\eta}) \arrow[uu,"u_{Y\eta \#}"]  \arrow[rr,"\Psi_{f \circ u \circ k }",dashed,pos=0.2] \arrow[rd,"\overline{u}_{\eta \#}",swap]  & & \sh(\overline{Y}_{\sigma}) \arrow[rd,"\overline{u}_{\sigma \#}",dashed]  \arrow[uu,"u_{Y \sigma \#}",pos=0.3,dashed]   & \\ 
            & \sh(\overline{X}_{\eta}) \arrow[uu,"u_{X\eta \#}",pos=0.3,crossing over]  \arrow[rr,"\Psi_{f \circ h}"] & & \sh(\overline{X}_{\sigma}) \arrow[uu,"u_{X\sigma  \#}"]  \arrow[from=2-2,to=2-4,"\Psi_{f}",pos=0.2,crossing over] \arrow[from=4-2,to=4-4,"\Psi_{f \circ \hat{h}}",pos=0.2,crossing over]
            \end{tikzcd}
   \end{equation*}
   in which the upper cube is commutative thanks to proposition \ref{magic diagram} and the lower cube is commutative because it involves only one type of base change structure. 
\end{proof}

\subsection{More on the exceptional direct image operation} \label{exceptional direct image} Let us explain how one can construct $g_!$ for a morphism $g \colon Y \longrightarrow X$ of $k$-varieties. We observe that the base change morphism $\Ex_{\#*}$ can be defined without the cartesian property but requiring some assumptions on morphisms. More concretely, given a commutative square 
\begin{equation*}
    \begin{tikzcd}[sep=large]
         X' \arrow[d,"f'",swap] \arrow[dr,phantom,"\textnormal{(C)}"] \arrow[r,"g'"] & X \arrow[d,"f"] \\ 
            Y' \arrow[r,"g"] & Y
    \end{tikzcd}
\end{equation*}
with $g,g'$ open immersions and $f,f'$ proper, we form a commutative diagram
\begin{equation*}
    \begin{tikzcd}[sep=large]
   X' \arrow[drr,"g'",bend left = 20] \arrow[dr,"\Phi"] \arrow[ddr,"f'",swap,bend right = 20] & & \\ 
         & Z \arrow[dr,phantom,"\textnormal{(C')}"] \arrow[r,"g^{''}"] \arrow[d,"f^{''}",swap] & X \arrow[d,"f"] \\ 
          &   Y' \arrow[r,"g"] & Y
    \end{tikzcd}
\end{equation*}
where $\textnormal{(C')}$ is cartesian. From the cartesian case, we have $\Ex_{\#*}(\Delta') \colon g_{\#}f^{''}_* \longrightarrow f_*g^{''}_{\#}$. The morphism $\Phi$ is both open and closed so the canonical transformation $\gamma \colon \Phi_{\#} \longrightarrow \Phi_*$ (obtained by the adjunction of the inverse of the counit $\Phi^*\Phi_* \overset{\sim}{\longrightarrow} \id$) is an isomorphism. Therefore we have a base change morphism
\begin{equation*}
    \Ex_{\#*}(\textnormal{C}) \colon g_{\#}f'_* = g_{\#}f^{''}_*\Phi_* \overset{\Ex_{\#*}(\textnormal{C'}}{\longrightarrow} f_*g^{''}_{\#} \Phi_* \overset{\gamma^{-1}}{\longrightarrow} f_*g^{''}_{\#} \Phi_{\#} = f_*g'_{\#}.
\end{equation*}
In particular, given a commutative diagram
\begin{equation*}
    \begin{tikzcd}[sep=large]
        Y \arrow[r,"u"] \arrow[d,equal] & \overline{Y} \arrow[d,"\Phi"] \\ 
        Y \arrow[r,"w"] & \hat{Y}.
    \end{tikzcd}
\end{equation*}
where $u,w$ are open immersions, $\Phi$ is a closed and open immersion, then there is an isomorphism $\Ex_{\#*}(\Delta) \colon w_{\#} \overset{\sim}{\longrightarrow} \Phi_*u_{\#}$ and because of its importance, we shall denote by $\mathrm{BC}^{\Phi}\colon w_{\#} \overset{\sim}{\longrightarrow} \Phi_*u_{\#}$ this morphism. By the Nagata compactification theorem, there exists a factorization 
\begin{equation*}
    \begin{tikzcd}[sep=large]
        Y \arrow[r,"u"] \arrow[rd,"g",swap] & \overline{Y} \arrow[d,"\overline{g}"] \\ 
        & X
    \end{tikzcd}
\end{equation*}
where $u$ is a quasi-compact open immersion and $\overline{g}$ is proper. We define $g_! \coloneqq \overline{g}_*u_{\#}$ and by definition we have a natural transformation $g_! \longrightarrow g_*$ induced by the canonical transformation $u_{\#}\longrightarrow u_*$. This definition does not depend on the choice of the compactification in the following sense: the category of compactifications of $g$ is filtered, so given two compactifications of $g$, one dominates another 
 \begin{equation*} 
    \begin{tikzcd}[sep=large]
    &  & \hat{Y} \arrow[ldd,bend left = 20,"\hat{g}"]  \\
        Y \arrow[r,"u"] \arrow[rru,bend left = 20,"w"] \arrow[rd,"g",swap] & \overline{Y} \arrow[d,"\overline{g}"] \arrow[ru,"\Phi"] & \\ 
        & X  &
     \end{tikzcd}
\end{equation*}
then there is a canonical isomorphism 
\begin{equation*}
    \mathrm{BC}^{\Phi} \colon \hat{g}_*w_{\#} \overset{\sim}{\longrightarrow} \overline{g}_*u_{\#}.
\end{equation*}
Let $h \colon Z \longrightarrow Y$ be another morphism of $k$-varieties, we then have a connection isomorphism $g_!h_! \overset{\sim}{\longrightarrow} (g \circ h)_!$ as follows.  Let  
    \begin{equation*}\begin{tikzcd}[sep=large]
        Z \arrow[r,"\overline{u}"] \arrow[rd,"h",swap] & \overline{Z} \arrow[d,"\overline{h}"] \\ 
        & Y.
    \end{tikzcd}
    \end{equation*}
    be a compactification of $h$. Note that $u$ is always quasi-compact, and hence separated of finite type. Consequently, $u \circ \overline{h}$ is compactifyable. There exists a commutative diagram 
    \begin{equation*}
    \begin{tikzcd}[sep=large]
        Z \arrow[r,"\overline{u}"] \arrow[dr,"h",swap] & \overline{Z} \arrow[dr,phantom,"\textnormal{(C)}"] \arrow[d,"\overline{h}",swap] \arrow[r,"\hat{u}"] & \hat{Z} \arrow[d,"\hat{h}"] \\ 
         & Y \arrow[r,"u"] \arrow[dr,"g",swap] & \overline{Y} \arrow[d,"\overline{g}"] \\ 
       &   & X.
    \end{tikzcd}
    \end{equation*}
    where $\hat{h}$ is proper and $\hat{u}$ is an open immersion. Then the desired isomorphism is given by
    \begin{equation*}
        g_! h_! = \overline{g}_*u_{\#}\overline{h}_*\overline{u}_{\#} \overset{\Ex_{\#*}(\textnormal{C})}{\longrightarrow} \overline{g}_* \hat{h}_*\hat{u}_{\#}\overline{u}_{\#} = (\overline{g} \circ \hat{h})_*(\hat{u}\circ \overline{u})_{\#} = (g \circ h)_!.
    \end{equation*}

\subsection{Base change of nearby cycles functors with respect to exceptional direct images} Keep the notation of the preceding section. In this section, we construct a base change morphism 
\begin{equation*}
\begin{tikzcd}[sep=large]
         \sh(Y_{\eta}) \arrow[r,"\Psi_{f \circ g}"] \arrow[d,"g_{\eta!}",swap] & \sh(Y_{\sigma}) \arrow[Rightarrow, shorten >=27pt, shorten <=27pt, dl,"\mu_g"] \arrow[d,"g_{\sigma!}"] \\ 
         \sh(X_{\eta}) \arrow[r,"\Psi_{f}"] &  \sh(X_{\sigma}).
    \end{tikzcd} 
    \end{equation*}
for varieties (not necessarily quasi-projective varieties) and $\mu_g$ becomes an isomorphism if $g$ is proper. Now we come to the construction of $\mu_g \colon g_{\sigma!}\Psi_{f \circ g} \longrightarrow \Psi_f g_{\eta!}$. By the Nagata compactification theorem, there exists a diagram 
\begin{equation*}
    \begin{tikzcd}[sep=large]
        Y \arrow[r,"u"] \arrow[rd,"g",swap] & \overline{Y} \arrow[d,"\overline{g}"] \\ 
        & X
    \end{tikzcd}
\end{equation*}
where $u$ is a quasi-compact open immersion and $\overline{g}$ is proper. We define
\begin{equation*}
    \mu_u \coloneqq \gamma_u \colon u_{\sigma \#}\Psi_{f \circ \overline{g} \circ u} \longrightarrow \Psi_{f \circ \overline{g}}u_{\eta \#}.
\end{equation*}
We apply $\overline{g}_{\sigma*}$ to both sides and remind that $g_{\sigma!} = \overline{g}_{\sigma*}u_{\sigma \#}$
\begin{equation*}
   \mu_u\colon g_{\sigma!}\Psi_{f \circ \overline{g} \circ u} \longrightarrow g_{\sigma*}\Psi_{f \circ \overline{g}}u_{\eta \#}.
\end{equation*}
By using lemma \ref{proper base change for nearby cycles functor}, we see that $\beta^{-1}_{\overline{g}}\colon \overline{g}_{\sigma *}\Psi_{f \circ \overline{g}} \overset{\sim}{\longrightarrow} \Psi_f \overline{g}_{\eta*}$ is an isomorphism. We define the base change morphism 
\begin{equation*}
    \mu_g  \coloneqq \beta_{\overline{g}}^{-1} \circ \mu_u \colon g_{\sigma!}\Psi_{f \circ g} \longrightarrow \Psi_f g_{\eta!}
\end{equation*}
which is clearly an isomorphism provided that $g$ is proper. In diagram, $\mu_g$ is the composition 
\begin{equation*}
    \begin{tikzcd}[sep=large]
         \sh(Y_{\eta}) \arrow[r,"\Psi_{f \circ g}"] \arrow[d,"u_{\eta \#}",swap] & \sh(Y_{\sigma}) \arrow[Rightarrow, shorten >=27pt, shorten <=27pt, dl,"\gamma_g"] \arrow[d,"u_{\sigma \#}"] \\ 
         \sh(\overline{Y}_{\eta}) \arrow[r,"\Psi_{f \circ \overline{g}}"] \arrow[d,"\overline{g}_{\eta *}",swap] & \sh(\overline{Y}_{\sigma}) \arrow[Rightarrow, shorten >=27pt, shorten <=27pt, dl,"(\beta_{\overline{g}})^{-1}"] \arrow[d,"\overline{g}_{\sigma*}"] \\ 
         \sh(X_{\eta}) \arrow[r,"\Psi_{f}"] &  \sh(X_{\sigma})
    \end{tikzcd}
\end{equation*}
Let us prove that $\mu_g$ is independent of the choice of the compactification. Since the category of compactifications is filtered, it is sufficient to prove the following
\begin{lem}
Let $X,Y$ be $k$-varieties and $g \colon Y \longrightarrow X$, $f \colon X \longrightarrow \mathbb{A}_k^1$ be morphisms of $k$-varieties. Given two compactifications of $g$, one dominates another
 \begin{equation*} 
    \begin{tikzcd}[sep=large]
    &  & \hat{Y} \arrow[ldd,bend left = 20,"\hat{g}"]  \\
        Y \arrow[r,"u"] \arrow[rru,bend left = 20,"w"] \arrow[rd,"g",swap] & \overline{Y} \arrow[d,"\overline{g}"] \arrow[ru,"\Phi"] & \\ 
        & X  &
     \end{tikzcd}
\end{equation*}
then we have a commutative diagram 
\begin{equation*}
    \begin{tikzcd}[sep=large]
         \hat{g}_{\sigma*}w_{\sigma \#}\Psi_{f \circ \hat{g} \circ w} \arrow[r," \beta_{\hat{g}}^{-1} \circ \gamma_w"] \arrow[d,"\mathrm{BC}_{\sigma}^{\Phi}",swap] & \Psi_f \hat{g}_{\eta*}w_{\eta \#} \arrow[d,"\mathrm{BC}_{\eta}^{\Phi}"] \\ 
       \overline{g}_{\sigma*}u_{\sigma \#} \Psi_{f \circ \overline{g} \circ u} \arrow[r," \beta_{\overline{g}}^{-1} \circ \gamma_u"] & \Psi_f \overline{g}_{\eta*}u_{\eta \#}.
    \end{tikzcd}
\end{equation*}
In other words, the base change morphism $\mu_g$ is independent of the choice of the  compactification. 
\end{lem}
\begin{proof}
We decompose the diagram into two others 
\begin{equation*}
    \begin{tikzcd}[sep=large]
         \hat{g}_{\sigma*}w_{\sigma \#}\Psi_{f \circ \hat{g} \circ w} \arrow[r," \gamma_w"] \arrow[d,"\mathrm{BC}_{\sigma}^{\Phi}",swap] &  \hat{g}_{\sigma*}\Psi_{f \circ \hat{g}}w_{\eta \#}  \arrow[d,"\beta_{\Phi}\mathrm{BC}^{\Phi}_{\eta}"] \\ 
       \overline{g}_{\sigma*}u_{\sigma \#} \Psi_{f \circ \overline{g} \circ u} \arrow[r,"\gamma_u"] & \overline{g}_{\sigma*}\Psi_{f \circ \overline{g}}u_{\eta \#}
    \end{tikzcd} \ \ \ \ \ \ \ \ 
    \begin{tikzcd}[sep=large]
           \hat{g}_{\sigma*}\Psi_{f \circ \hat{g}}w_{\eta \#}  \arrow[d,"\beta_{\Phi}\mathrm{BC}^{\Phi}_{\eta \#}",swap] \arrow[r,"\beta_{\hat{g}}^{-1}"] & \Psi_f \hat{g}_{\eta*}w_{\eta \#}  \arrow[d,"\mathrm{BC}^{\Phi}_{\eta}"] \\ 
       \overline{g}_{\sigma*}\Psi_{f \circ \overline{g}}u_{\eta \#} \arrow[r,"\beta_{\overline{g}}^{-1}"] & \Psi_f \overline{g}_{\eta*}u_{\eta \#}.
    \end{tikzcd}
\end{equation*}
Let us prove that the right diagram is commutative. We break it again into 
\begin{equation*}
        \begin{tikzcd}[sep=large]
          \Psi_f \hat{g}_{\eta*}w_{\eta \#}  \arrow[d,"\mathrm{BC}^{\Phi}_{\eta}",swap] \arrow[r,"\beta_{\hat{g}}"]   & \hat{g}_{\sigma*}\Psi_{f \circ \hat{g}}w_{\eta \#}  \arrow[d,"\mathrm{BC}^{\Phi}_{\eta}"]   \\ 
      \Psi_f \hat{g}_{\eta*} \Phi_{\eta *}u_{\eta \#} \arrow[r,"\beta_{\hat{g}}"] \arrow[rd,"\beta_{\overline{g}}",swap]  &  \hat{g}_{\sigma*}\Psi_{f \circ \hat{g}}\Phi_{\eta *}u_{\eta \#}   \arrow[d,"\beta_{\Phi}"] \\ 
      & \overline{g}_{\sigma*}\Psi_{f \circ \overline{g}}u_{\eta \#}.
    \end{tikzcd}
\end{equation*}
The low triangle is commutative because of the compatibility of base change morphisms $\beta_?$. The square is commutative thanks to the interchange law of composition of natural transformations. The left diagram is decomposed into 
\begin{equation*}
        \begin{tikzcd}[sep=large]
         w_{\sigma \#}\Psi_{f \circ \hat{g} \circ w} \arrow[r," \mu_w"] \arrow[d,"\mathrm{BC}_{\sigma}^{\Phi}",swap] &  \Psi_{f \circ \hat{g}}w_{\eta \#}  \arrow[d,"\mathrm{BC}^{\Phi}_{\eta}"] \\ 
       \Phi_{\sigma*}u_{\sigma \#}  \Psi_{f \circ \overline{g} \circ u} \arrow[rd,"\mu_u",swap] & \Psi_{f \circ \hat{g}}\Phi_{\eta *}u_{\eta \#} \arrow[d,"\beta_{\Phi}"] \\ 
      & \Phi_{\sigma*}\Psi_{f \circ \overline{g}}u_{\eta \#}.
    \end{tikzcd}
\end{equation*}
But this is a direct application of proposition \ref{magic diagram} to the diagram
\begin{equation*}
    \begin{tikzcd}[sep=large]
        Y \arrow[r,"u"] \arrow[d,equal] & \overline{Y} \arrow[d,"\Phi"] \\ 
        Y \arrow[r,"w"] & \hat{Y}.
    \end{tikzcd}
\end{equation*}
\end{proof}
The collection of base change morphisms $\mu_?$ satisfies the following compatibility. 

\begin{prop} \label{compatibility of nearby cycles functors with exceptional direct images}
 Let $X,Y,Z$ be $k$-varieties and $h \colon Z \longrightarrow Y$, $g\colon Y \longrightarrow X$, $f\colon X \longrightarrow \mathbb{A}_k^1$ be morphisms of $k$-varieties. Then the following diagram is commutative
\begin{equation*}
        \begin{tikzcd}[sep=large]
              (g \circ h)_{\sigma!}\Psi_{f \circ g \circ h} \arrow[d,"\mu_h",swap] \arrow[r,"\mu_{g \circ h}"] &  \Psi_{f}(g \circ h)_{\eta!}\arrow[d,"\sim"]    \\ 
              g_{\sigma!}\Psi_{f \circ g}h_{\eta !} \arrow[r,"\mu_g"] &  \Psi_{f}g_{\eta!}h_{\eta!}.
        \end{tikzcd}
    \end{equation*}    
\end{prop}
\begin{proof}
    Keep the notation of section \ref{exceptional direct image}, we need to show the commutativity of 
     \begin{equation*}
        \begin{tikzcd}[sep=large]
              & (\overline{g} \circ \hat{h})_{\sigma *}(\hat{u} \circ \overline{u})_{\sigma \#}\Psi_{f \circ g \circ h}  \arrow[d,"\Ex_{\#*}(\textnormal{C})^{-1}",swap]   \arrow[r,"\mu_{g \circ h}"] &  \Psi_{f} (\overline{g} \circ \hat{h})_{\eta *}(\hat{u} \circ \overline{u})_{\eta \#}  \arrow[dd,"\Ex_{\# *}(\textnormal{C})^{-1}"]    \\ 
              & \overline{g}_{\sigma*}u_{\sigma \#} \overline{h}_{\sigma *} \overline{u}_{\sigma \#}\Psi_{f \circ g \circ h} \arrow[d,"\mu_h",swap] \arrow[ld,swap,"\mu_w"] &   \\ 
            \overline{g}_{\sigma*}u_{\sigma \#} \overline{h}_{\sigma *}\Psi_{f \circ g \circ \overline{h}}\overline{u}_{\eta \#}   \arrow[r,"\beta_{\overline{h}}^{-1}"] &  \overline{g}_{\sigma*}u_{\sigma \#} \Psi_{f \circ g} \overline{h}_{\eta *} \overline{u}_{\eta \#}  \arrow[r,"\mu_g"]  & \Psi_{f}\overline{g}_{\eta*}u_{\eta \#}\hat{h}_{\eta *}\overline{u}_{\eta \#}. 
        \end{tikzcd}
    \end{equation*} 
    Since the diagram 
    \begin{equation*}
        \begin{tikzcd}[sep=large]
             (\overline{g} \circ \hat{h})_{\sigma *}(\hat{u} \circ \overline{u})_{\sigma \#}\Psi_{f \circ g \circ h} \arrow[d,"\mu_{\overline{u}}",swap] \arrow[r,"\Ex_{\#*}(\textnormal{C})^{-1}"] &   \overline{g}_{\sigma*}u_{\sigma \#} \overline{h}_{\sigma *} \overline{u}_{\sigma \#}\Psi_{f \circ g \circ h}  \arrow[d,"\mu_{\overline{u}}"] \\ 
               \overline{g}_{\sigma*}\hat{h}_{\sigma *}\hat{u}_{\sigma \#}\Psi_{f \circ g \circ \overline{h}}\overline{u}_{\eta \#} \arrow[r,"\Ex_{\#*}(\textnormal{C})^{-1}"] & \overline{g}_{\sigma*}u_{\sigma \#} \overline{h}_{\sigma *}\Psi_{f \circ g \circ \overline{h}} \overline{u}_{\eta \#}.
        \end{tikzcd}
    \end{equation*}
    is commutative by the interchange law, the previous diagram can be decomposed into a bigger diagram 
    \begin{equation*}
        \begin{tikzcd}[sep=large]
              (\overline{g} \circ \hat{h})_{\sigma *}(\hat{u} \circ \overline{u})_{\sigma \#}\Psi_{f \circ g \circ h}  \arrow[d,"\mu_{\overline{u}}",swap]   \arrow[r,"\mu_{\hat{u} \circ \overline{u}}"] & (\overline{g} \circ \hat{h})_{\sigma *}\Psi_{f \circ \overline{g} \circ \hat{h}} (\hat{u} \circ \overline{u})_{\eta \#} \arrow[r,"\beta_{\overline{g}\circ \hat{h}}^{-1}"] \arrow[d,"\beta_{\hat{h}}^{-1}"] & \Psi_{f} (\overline{g} \circ \hat{h})_{\eta *}(\hat{u} \circ \overline{u})_{\eta \#}  \arrow[ddd,"\Ex_{\# *}(\textnormal{C})^{-1}"]    \\ 
              \overline{g}_{\sigma*}\hat{h}_{\sigma *}\hat{u}_{\sigma \#}\Psi_{f \circ \overline{g} \circ \hat{h} \circ \hat{u}}\overline{u}_{\eta \#} \arrow[rdd,phantom,"\textnormal{(C1)}"] \arrow[ru,"\mu_{\hat{u}}",swap] \arrow[d,"\Ex_{\#*}(\Delta)^{-1}",swap] & \overline{g}_{\sigma*}\Psi_{f \circ \overline{g}}\hat{h}_{\eta *}(\hat{u} \circ \overline{u})_{\eta \#} \arrow[ddr,phantom,"\textnormal{(C2)}"]  \arrow[ru,"\beta_{\overline{g}}^{-1}",swap] \arrow[dd,"\Ex_{\#*}(\Delta)^{-1}"]&  \\ 
               \overline{g}_{\sigma*}u_{\sigma \#} \overline{h}_{\sigma *}\Psi_{f \circ \overline{g} \circ u \circ \overline{h}} \overline{u}_{\eta \#}\arrow[d,"\beta_{\overline{h}}^{-1}",swap] &  &  \\ 
                \overline{g}_{\sigma*}u_{\sigma \#}\Psi_{f \circ \overline{g} \circ u} \overline{h}_{\eta *} \overline{u}_{\eta \#}\arrow[r,"\mu_u"] &   \overline{g}_{\sigma*} \Psi_{f \circ \overline{g}} u_{\eta \#}\overline{h}_{\eta *} \overline{u}_{\eta \#}  \arrow[r,"\beta_{\overline{g}}^{-1}"]  & \Psi_{f}\overline{g}_{\eta *} u_{\eta \#}\overline{h}_{\eta *} \overline{u}_{\eta \#}.
        \end{tikzcd}
    \end{equation*} 
    In this diagram, the two upper triangles are commutative by definition. The commutativity of the square trapezoid $\textnormal{(C2)}$ follows from the interchange law. It remains to prove the commutativity of the square trapezoid $\textnormal{(C1)}$. After removing common terms, $\textnormal{(C1)}$ becomes 
    \begin{equation*}
        \begin{tikzcd}[sep=large]
            \hat{h}_{\sigma *}\hat{u}_{\sigma \#}\Psi_{f \circ \overline{g} \circ \hat{h} \circ \hat{u}} \arrow[r]  &  \hat{h}_{\sigma *}\Psi_{f \circ \overline{g} \circ \hat{h}}\hat{u}_{\eta \#} \\ 
             u_{\sigma \#} \overline{h}_{\sigma *}\Psi_{f \circ \overline{g} \circ u \circ \overline{h}} \arrow[u,"(\Ex_{\#*})"]   & \Psi_{f \circ \overline{g}} \hat{h}_{\eta *}\hat{u}_{\eta \#}  \arrow[u] \\ 
             u_{\sigma \#}\Psi_{f \circ \overline{g} \circ u}\overline{h}_{\eta *}\arrow[r] \arrow[u] &   \Psi_{f \circ \overline{g}}u_{\eta \#}\overline{h}_{\eta *}. \arrow[u,"(\Ex_{\#*})",swap]
        \end{tikzcd}
    \end{equation*}
    But this is an immediate application of proposition \ref{magic diagram} to the diagram $\textnormal{(C)}$. 
\end{proof}

We also have a compatibility between $\mu_?$ and $\beta_?$. The following corollary is a generalization of \cite[Proposition 3.1.10]{ayoub-thesis-2} to varieties (not necessarily quasi-projective). 
\begin{cor}
    Let $X,Y$ be $k$-varieties and $g\colon Y \longrightarrow X$, $f\colon X \longrightarrow \mathbb{A}_k^1$ be morphisms of $k$-varieties. Then the following diagram is commutative
    \begin{equation*}
        \begin{tikzcd}[sep=large]
            g_{\sigma!}\Psi_{f \circ g} \arrow[r,"\mu_g"] \arrow[d] & \Psi_f g_{\eta !}\arrow[d] \\ 
            g_{\sigma*}\Psi_{f \circ g} & \Psi_f g_{\eta *}. \arrow[l,"\beta_g"] 
        \end{tikzcd}
    \end{equation*}
\end{cor}

\begin{proof}
    Let 
    \begin{equation*}
        \begin{tikzcd}[sep=large]
             Y \arrow[r,"u"] \arrow[rd,"g",swap] & \overline{Y} \arrow[d,"\overline{g}"] \\ 
        & X
        \end{tikzcd}
    \end{equation*}
    be a compactification of $g$. The commutativity of the diagram above follows from the commutativity of
    \begin{equation*}
        \begin{tikzcd}[sep=large]
            \overline{g}_{\sigma*}u_{\sigma \#}\Psi_{f \circ g} \arrow[dr,phantom,"\textnormal{(C1)}"] \arrow[r,"\mu_g"] \arrow[d] &  \overline{g}_{\sigma*}\Psi_{f \circ \overline{g}}u_{\eta \#} \arrow[dr,phantom,"\textnormal{(C2)}"] \arrow[r,"\beta^{-1}_{\overline{g}}"] \arrow[d] & \Psi_f \overline{g}_{\eta *}u_{\eta \#} \arrow[d] \\ 
            \overline{g}_{\sigma*}u_{\sigma *}\Psi_{f \circ g} & \overline{g}_{\sigma*}\Psi_{f \circ \overline{g}}u_{\eta *} \arrow[l,"\beta_u"] & \Psi_f \overline{g}_{\eta *}u_{\eta *}. \arrow[l,"\beta_{\overline{g}}"] 
        \end{tikzcd}
    \end{equation*}
    By the interchange law, the square $\textnormal{(C2)}$ is commutative. The commutativity of $\textnormal{(C1)}$ boils down to the commutativity of
     \begin{equation*}
        \begin{tikzcd}[sep=large]
           u_{\sigma \#}\Psi_{f \circ g} \arrow[dr,phantom,"\textnormal{(C1)}"] \arrow[r,"\mu_g"] \arrow[d] &  \Psi_{f \circ \overline{g}}u_{\eta \#}  \arrow[d]  \\ 
          u_{\sigma *}\Psi_{f \circ g} & \Psi_{f \circ \overline{g}}u_{\eta *}.\arrow[l,"\beta_u"]
        \end{tikzcd}
    \end{equation*}
    But $u_{\sigma \#}$ and $u_{\sigma*}$ are fully faithful, so it suffices to prove that $\textnormal{(C1)}$ is commutative after we apply $u_{\sigma}^*$. The diagram becomes 
     \begin{equation*}
        \begin{tikzcd}[sep=large]
         \Psi_{f \circ g} \arrow[r] \arrow[d] &  u_{\sigma}^*\Psi_{f \circ \overline{g}}u_{\eta \#}  \arrow[d]  \\ 
          \Psi_{f \circ g} & u^*_{\sigma} \Psi_{f \circ \overline{g}}u_{\eta *} \arrow[l]
        \end{tikzcd}
    \end{equation*}
    which is commutative for trivial reason. 
\end{proof}

\bibliographystyle{alpha}
\bibliography{ref}

\end{document}